\documentclass[12pt]{amsart}

\usepackage{amsthm,amscd,amsfonts}
\usepackage{amssymb, upref}
\usepackage{latexsym}
\usepackage{amsmath,amsthm,amsfonts,amssymb}
\usepackage[pdftex]{graphicx}

\newtheorem{theorem}{Theorem}[section]
\newtheorem{lemma}[theorem]{Lemma}
\newtheorem{corollary}[theorem]{Corollary}
\newtheorem{proposition}[theorem]{Proposition}
\theoremstyle{definition}
\newtheorem{definition}[theorem]{Definition}
\newtheorem{remark}[theorem]{Remark}
\newtheorem{example}[theorem]{Example}
\def\bal{\begin{array}{ll}}
\def\eal{\end{array}}
\def\S{\mathcal S}
\numberwithin{equation}{section}

\def\Young#1{\vbox{\smallskip\offinterlineskip
    \halign{&\vbox{##}\kern-\Thickness\cr #1}}}

\newdimen\Squaresize \Squaresize=24pt
\newdimen\Thickness \Thickness=.1pt
\newdimen\Correction \Correction=7pt

\def\Vide#1{\hbox{
       \vbox to \Squaresize{\vss
          \hbox to \Squaresize{\hss#1 \hss}\vss}
    \hskip-\Correction}
   \kern-\Thickness}

\def\Carre#1{\hbox{\vrule width \Thickness
   \vbox to \Squaresize{\hrule height \Thickness\vss
      \hbox to \Squaresize{\hss$\scriptstyle#1$\hss}
   \vss\hrule height\Thickness}
   \unskip\vrule width \Thickness}
   \kern-\Thickness}

\def\Tabvrule{\vrule width-0.7pt}       
\def\Tabhrule{\hrule \hrule height-0.7pt} 
\def\Tabstrut{\vrule height2.3ex 
                     depth0.9ex  
                     width0ex    
\relax}

\def\PasCase#1{\omit%
            $\vcenter{\hbox {\vbox to 0.7pt{}}
               \hbox{\makebox[3ex]{\Tabstrut$#1$}}}%
               \Tabvrule$}
\def\PasCasePoint{\PasCase{\cdot}}
\def\DessinCarre#1{%
    \vcenter{\hbox{}\hrule
             \hbox{\vrule\makebox[3ex]{\Tabstrut$#1$}\vrule}\Tabhrule}%
             \Tabvrule}
\def\GenRuban#1{\vcenter{\halign{&$\DessinCarre{##}$\cr#1}}\egroup}

\def\sTabvrule{\vrule width-0.7pt}
\def\sTabhrule{\hrule \hrule height-0.7pt}
\def\sTabstrut{\vrule height1.7ex depth0.7ex width0ex \relax}
\def\sDessinCarre#1{%
    \vcenter{\hbox{}\hrule
             \hbox{\vrule\makebox[2.4ex]%
                  {\sTabstrut$\scriptstyle#1$}\vrule}\sTabhrule}%
             \sTabvrule}
\def\sGenRuban#1{\vcenter{\halign{&$\sDessinCarre{##}$\cr#1}}\egroup}

\def\ruban{%
  \bgroup
  \let\ =\omit
  \let\\=\cr
  \let\.=\PasCasePoint
  \offinterlineskip
  \GenRuban}

\title[Bases for Alternating Harmonic Polynomials]{Bases for Diagonally Alternating Harmonic Polynomials of low degree}

\author[N. Bergeron and Z. Chen]
{Nantel Bergeron and Zhi Chen}

\address[Nantel Bergeron]
{Department of Mathematics and Statistics\\ York University\\
   To\-ron\-to, Ontario M3J 1P3\\ CANADA}
\email[Nantel Bergeron]{bergeron@mathstat.yorku.ca}
\urladdr[Nantel Bergeron]{http://www.math.yorku.ca/bergeron}

\address[Zhi Chen]
{Department of Mathematics and Statistics\\ York University\\
   To\-ron\-to, Ontario M3J 1P3\\ CANADA} 
\email[Zhi Chen]{czhi@mathstat.yorku.ca}
\date{\today}


\begin{document}

\begin{abstract} Given a list of $n$ cells $L=[(p_1,q_1),\ldots,(p_n, q_n)]$ where $p_i, q_i\in \textbf{Z}_{\ge 0}$, we let
 $\Delta_L=\det\left\|{(p_j!)^{-1}(q_j!)^{-1} x^{p_j}_iy^{q_j}_i}\right\|$. The space of diagonally alternating polynomials
 is spanned by $\{\Delta_L\}$ where $L$ varies among all lists with $n$ cells. For $a>0$, the operators $E_a=\sum_{i=1}^{n} y_i\partial_{x_i}^a$ 
 act on diagonally alternating polynomials. Haiman has shown that the space $A_n$ of diagonally alternating harmonic polynomials 
 is spanned by $\{E_\lambda\Delta_n\}$ where $\lambda=(\lambda_1,\ldots,\lambda_\ell)$ varies among all partitions, $E_\lambda=E_{\lambda_1}\cdots E_{\lambda_\ell}$ and $\Delta_n=\det\big\|{((n-j)!)^{-1} x^{n-j}_{i}}\big\|$. 
For $t=(t_m,\ldots,t_1)\in \textbf{Z}_{> 0}^m$  with $t_m>\cdots>t_1>0$, we consider here the operator $F_t=\det\big\|E_{t_{m-j+1}+(j-i)}\big\|$.
Our first result is to show that $F_t\Delta_L$ is a linear combination of $\Delta_{L'}$ where $L'$ is obtained by {\sl moving} $\ell(t)=m$ distinct cells of $L$ 
in some determined fashion. This allows us to control the leading term of some elements of the form $F_{t_{(1)}}\cdots F_{t_{(r)}}\Delta_n$. 
We use this to describe explicit bases of some of the bihomogeneous components of $A_n=\bigoplus A_n^{k,l}$ where $A_n^{k,l}=\hbox{Span}\{E_\lambda\Delta_n :\ell(\lambda)=l, |\lambda|=k\}$. More precisely, we give an explicit basis of $A_n^{k,l}$ whenever $k<n$. To this end, we introduce a new variation of Schensted insertion on a special class of tableaux. This produces a bijection between partitions and this new class of tableaux. The combinatorics of these tableaux $T$ allow us to know exactly the leading term of $F_T\Delta_n$ where $F_T$ is the operator corresponding to the columns of $T$, whenever $n$ is greater than the weight of $T$.

\end{abstract}

\maketitle
\section{Introduction}

The theory of Macdonald symmetric polynomials~\cite{Mac} is a very active and deep area of mathematics.
At the heart of this theory lies the study of $(q,t)$-Catalan numbers and the space of diagonal harmonics introduced by Garsia, Haiman and collaborators
(see~\cite{Jim} and references therein). The space over $\mathbb Q$ of diagonal harmonics in $2n$ variables is given by
  $$H_n=\{P\in {\mathbb Q}[X_n,Y_n]: \sum^n_{i=1}\partial_{x_i}^k\partial_{y_i}^h P=0, h+k>0\},$$
where $X_n=\{x_1,x_2,\ldots,x_n\}$ and $Y_n=\{y_1,y_2,\ldots,y_n\}$. One of the many fascinating properties of this space is that its dimension~\cite{Haim} is $(n+1)^{n-1}$. 

The symmetric group $\S_n$ acts diagonally on ${\mathbb Q}[X_n,Y_n]$. That is, for $P\in {\mathbb Q}[X_n,Y_n]$,  the action $\sigma\in \S_n$ is defined by 
$\sigma P=P(x_{\sigma(1)},\ldots,x_{\sigma(n)},y_{\sigma(1)},\ldots,y_{\sigma(n)})$. Since the defining equations of $H_n$ are all symmetric, it is clear that $H_n$ is an $\S_n$-module. We can then consider the subspace of $H_n$ consisting of alternating polynomials. That is
  $$A_n = \{ P\in H_n : \sigma P= (-1)^{\ell(\sigma)} P,\, \forall \sigma\in\S_n\},$$
where $\ell(\sigma)$ denotes the length of $\sigma$. The space of polynomials ${\mathbb Q}[X_n,Y_n]$ is bigraded in $\textbf{Z}_{\ge 0}^2$ using the total degree in the variables $X_n$  and the total degree in the variables  $Y_n$. Since the diagonal action of $\S_n$ on ${\mathbb Q}[X_n,Y_n]$ preserves both degrees in $X_n$ and $Y_n$, we have that $H_n$ is a bigraded $\S_n$-module and $A_n=\bigoplus_{k,l} A_n^{k,l}$ is an $\S_n$-submodule of $H_n$. Here
$A_n^{k,l}$ consists of the bihomogeneous polynomials in $A_n$ of total degree $\frac{n(n-1)}{2}-k$ in the variables  $X_n$ and total degree $l$ in the variables $Y_n$.
We have shifted the degree in $X_n$ to simplify the formulation of our theorems. The polynomial
  $$C_n(q,t)= q^{\frac{n(n-1)}{2}} \sum_{k,l} \dim(A_n^{k,l}) q^{-k}t^l\,,$$
is known as the $(q,t)$-Catalan polynomial~\cite{Haim, Jim}. In particular, the dimension of $A_n$ is the Catalan number $C_n=\frac{1}{n+1}\left( 2n \atop n\right)$.
  
Given a list of $n$ cells, $L=[(p_1,q_1),\ldots,(p_n, q_n)]$ where $p_i, q_i\in \textbf{Z}_{\ge 0}$, we let
  $$\Delta_L=\det\left\|{\frac{1}{p_j! q_j!} x^{p_j}_iy^{q_j}_i}\right\|.$$
The space of all diagonally alternating polynomials in ${\mathbb Q}[X_n,Y_n]$ has a basis given by 
 $\{\Delta_L\}$, where $L = L(D)$ varies among all sets $D\subset \textbf{Z}_{\ge 0}^2$ of cardinality $n$ and $L(D)$ is the elements of $D$ given in a sorted list.
For $a>0$, the operators 
  $$E_a=\sum_{i=1}^{n} y_i\partial_{x_i}^a$$ 
 act on diagonally alternating polynomials. Haiman~\cite{Haim}  has shown that the space $A_n$ 
 is spanned by $\{E_\lambda\Delta_n\}$, where $\lambda=(\lambda_1,\ldots,\lambda_\ell)$ varies among all partitions, $E_\lambda=E_{\lambda_1}\cdots E_{\lambda_\ell}$ and $\Delta_n=\Delta_{[(0,0),(1,0),\ldots,(n-1,0)]}$. We have that 
   $$A_n^{k,l}=\hbox{Span}\{E_\lambda\Delta_n :\ell(\lambda)=l, |\lambda|=k\},$$
where $\ell(\lambda)=l$ denotes the number of parts of $\lambda$ and $ |\lambda|=\lambda_1+\cdots+\lambda_l$.

For $t=(t_m,\ldots,t_1)\in  \textbf{Z}_{> 0}^m$  with $t_m>\cdots>t_1>0$, we consider the operator $F_t=\det\big\|E_{t_{m-j+1}+(j-i)}\big\|$.
Our first result (Theorem~\ref{thm:Fop}) is to show that $F_t\Delta_L$ is a linear combination of $\Delta_{L'}$, where $L'$ is obtained by {\sl moving} $\ell(t)=m$ distinct cells of $L$  in some determined fashion. Given a column-strict Young tableau $T$ we can associate to each column of $T$ an $F$-operator as above and define $F_T$ to be the operator obtained as the product of the $F$-operators corresponding to the columns of $T$.
For $k<n$, we show (Corollary~\ref{cor:basis}) that a basis of $A_n^{k,l}$ is given by $\{F_T\Delta_n\}$, where $T$ runs over certain column-strict Young tableaux.

 To this end, we introduce a new variation of Schensted insertion on a special class of tableaux (Section~\ref{sec:NZT} and Section~\ref{sec:ITS}). This produces a bijection $\lambda\leftrightarrow T(\lambda)$ between partitions and this new class of tableaux (Section~\ref{sec:bijec}). The combinatorics of the tableaux $T(\lambda)$ allow us to know exactly the leading term of $F_{T(\lambda)}\Delta_n$ whenever $n$ is larger than the weight of $T$. We believe that  the insertion algorithm presented here may be of interest on its own: Corollary~\ref{cor:basis} is just one application of our construction. We point out that it is possible to get Corollary~\ref{cor:basis} more directly but this is less revealing for us.
 
 We present two short sections to recall some facts about $(q,t)$--Catalan numbers (Section~\ref{sec:catalan}) and an ordering of the diagonally alternating  polynomials (Section~\ref{sec:order}).

\section{ $(q,t)$-Catalan}\label{sec:catalan}

In this section we recall some of the basic definitions related to $(q,t)$--Catalan numbers.
\begin{definition}
A {\sl Dyck path} of length $n$ is a lattice path from the point $(0,0)$ to the point $(n,n)$ consisting of $n$ north steps $(0,1)$ and $n$ east steps $(1,0)$, that never cross the line $y=x$. The $i^{th}$ row of a Dyck path lies between the line $y=i-1$ and $y=i$. 
\end{definition}
We denote by $DP_n$, the set of all the Dyck paths of length $n$. 
Dyck paths of length $n$ are in bijection with sequences $g=(g_0,\ldots, g_{n-1})$ of $n$ nonnegative integers satisfying the two conditions
\begin{equation}
\left \lbrace
\begin{array}{lcc}
 g_0 = 0 \ ,& \\
0\le  g_{i+1} \le g_i+1\ , & \forall i < n-1 \ .
\end{array}
\right .
\end{equation}
The $i^{th}$ entry $g_{i-1}$ of the sequence $g$ corresponds to the number of complete lattice squares between the north step of the $i^{th}$ row of the Dyck path and the diagonal $y=x$. Such sequences are called Dyck sequences.
\begin{definition}
Given a Dyck path $c\in DP_n$, let $(g_0,\ldots, g_{n-1})$ be its corresponding Dyck sequence. The {\sl area} and {\sl coarea} of the Dyck path are given by
  $$a(c)=\sum_{i=0}^{n-1} g_i \qquad\hbox{and}\qquad ca(c)=\sum_{i=0}^{n-1} i- g_i =\frac{n(n-1)}{2} - a(c)$$
respectively. The {\sl bounce} statistic of the Dyck path is defined recursively as follows:
  $$b(c)=b(g_0,\ldots, g_{n-1})=n-1-g_{n-1} + b(g_0,\ldots,g_{n-2-g_{n-1}}),$$
where for the empty sequence $\epsilon=()$, we let $b(\epsilon)=0$.
\end{definition}
\begin{example} \label{ex:dyck}
The Dyck sequence $g=(0,0,1,2,0,1,1,2,3,0)$ corresponds to the following Dyck path $c$
\begin{center}
\vskip -12pt
\includegraphics[angle=270, width=3cm]{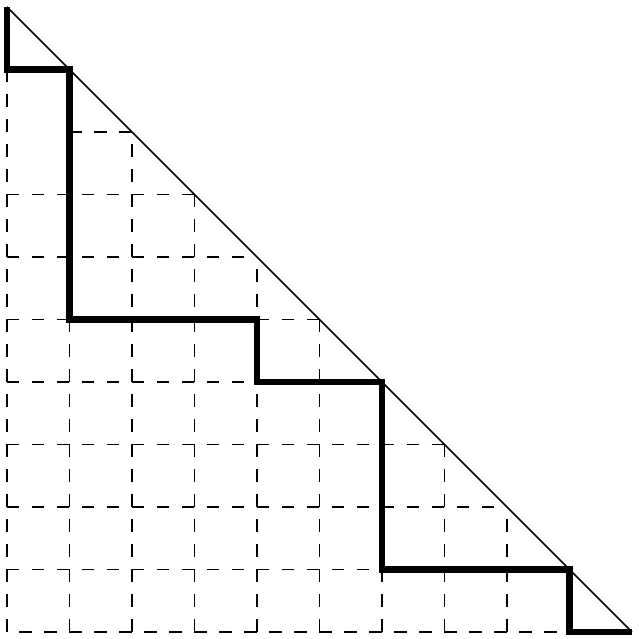}
\end{center}
The area and coarea of this Dyck path are $a(c)=1+2+1+1+2+3=10$ and $ca(c)=45-10=35$. The bounce statistic of $c$ is given by
\begin{align*}
  b(c)&=9+b(0,0,1,2,0,1,1,2,3)=9+5+b(0,0,1,2,0)=9+5+4+b(0,0,1,2)\\
   &=9+5+4+1+b(0)=9+5+4+1+0=19.\\
\end{align*}
\end{example}
\begin{remark}\label{rem:lambda}
A {\sl partition} $\mu$ of $m\in  \textbf{Z}_{>0}$, denoted by $\mu\vdash m$, is a sequence 
$\mu=(\mu_1, \mu_2, \ldots, \mu_l)$
of positive integers in non-increasing order:
$\mu_1\geq\mu_2\geq\cdots\geq\mu_l$
and $|\mu|:=\sum_{i=1}^{l}\mu_i=m$. 
The coarea of a Dyck path corresponds to the size of the partition $\lambda^t=\mu=(\mu_1,\ldots,\mu_{n-1})$ defined by $\mu_{i}=n-i-g_{n-i}$.
Here $\lambda^t$ denotes the transpose of the partition $\lambda$.
In the Example~\ref{ex:dyck}, $\mu=(9,5,5,5,4,4,1,1,1)$ and $\lambda=\mu^t=(9,6,6,6,4,1,1,1,1)$ are partitions of size $35$. The transpose here is the reflection of $\mu$
across the anti-diagonal.
\end{remark}

Let
\begin{equation}\label{eq:catalan}
  \widetilde{C}_n(q,t)=q^{\frac{n(n-1)}{2}} C_n(q^{-1},t)= \sum_{k,l} \dim(A_n^{k,l}) q^kt^l\,.
\end{equation}
A result by Garsia and Haglund~\cite{GHag,Jim} gives that
  $$C_n(q,t)=\sum_{c\in DP_n} q^{a(c)}t^{b(c)}.$$
In particular,
\begin{equation}\label{eq:qtcatalan}
  \widetilde{C}_n(q,t)=\sum_{c\in DP_n} q^{ca(c)}t^{b(c)}.
\end{equation}

\section{Sorting and ordering of diagonally  alternating polynomials}\label{sec:order}

In this section, we give a basis of diagonally alternating polynomials. Using an order on this basis we define a notion of leading term for any diagonally alternating polynomial.

Given a set of $n$ distinct cells $D=\{(p_1,q_1),\ldots,(p_n, q_n)\}\subset \textbf{Z}^2$ we say that the cells are {\sl sorted} if for all
$i<j$ we have that $q_i<q_j$ or ($q_i=q_j$ and $p_i<p_j$). We let $L(D)=[(p_1,q_1),\ldots,(p_n, q_n)]\in (\textbf{Z}^2)^n$ denote the sorted list.
On the other hand, if we are given a list of $n$ cells $L=[(p_1,q_1),\ldots,(p_n, q_n)]\in (\textbf{Z}^2)^n$ and if all $p_i,q_i\ge 0$,  we let
    $$\Delta_L=\det\left\|{\frac{1}{p_j! q_j!} x^{p_j}_iy^{q_j}_i}\right\|.$$
Otherwise we let $\Delta_L=0$. Notice that if the cells of $L$ are not distinct, then we also get $\Delta_L=0$. 
We call a list of $n$ cells $L=[(p_1,q_1),\ldots,(p_n, q_n)]\in (\textbf{Z}^2)^n$ a {\sl lattice diagram} .

For a lattice diagram $L=[(p_1,q_1),\ldots,(p_n, q_n)]$, let $\overline{L}=L(\{(p_1,q_1),\ldots,(p_n, q_n)\})$. That is $\overline{L}$ is the  list $L$ sorted. In particular we have
  $$\Delta_{\overline{L}} = \pm \Delta_L,$$
where the sign is determined by the sign of the permutation that reorders $L$ into $\overline{L}$. A basis of diagonally alternating polynomials in ${\mathbb Q}[X_n,Y_n]$ is given by the set
\begin{equation}\label{eq:basis}
  \big\{ \Delta_{L(D)} : D\subset \textbf{Z}_{\ge 0}^2\ \hbox{ and }\ |D|=n\big\}\,.
\end{equation}

Two sorted lattice diagrams $L(D)=[(p_1,q_1),\ldots,(p_n, q_n)]$ and $L(D')=$\break $[(p'_1,q'_1),\ldots,(p'_n, q'_n)]$
can be compared using the following {\sl lexicographic} order: 
\begin{equation}\label{eq:lexiorder}
   L(D)\prec L(D')\ \Leftrightarrow \quad  \exists i \left\{\begin{array}{l} ({p}_s,{q}_s) = ({p}'_s,{q}'_s), \quad  i+1\leq s\leq n,\\
                   ({p}_i,{q}_i) < ({p}'_i,{q}'_i) \,. \end{array}\right. 
\end{equation}
Given a diagonally alternating
 polynomial $f(X_n;Y_n)=a_1\Delta_{{L}(D_1)}+a_2\Delta_{{L}(D_2)}+\cdots+a_r\Delta_{{L}(D_r)}$ with all $a_i\ne 0$,  we define the {\sl leading diagram} of $f(X_n;Y_n)$ to be $\Delta_{{L}(D_k)}\ne 0$, where ${L}(D_k)\succ {L}(D_i)$ for all  $i\neq k$ and $1\leq i\leq r$.

\section{F-Operators}
In this section we introduce the operator $F$ for a column and show its basic properties.
A {\sl composition} $a$ of $n$, denoted by $a\models n$, is an ordered sequence of positive integers $a = (a_1,a_2,\ldots, a_k)$ such that $|a |:=\sum_i a_i=n$. For $a\models n$ we let $E_{a}=E_{a_1}E_{a_2}\cdots E_{a_k}$. Let $S_k$ denote the symmetric group on $k$ elements and let
\begin{align*}
\alpha: S_k &\longmapsto\textbf{Z}^k\\
        w   &\longmapsto \alpha(w)=(\alpha_1(w),\alpha_2(w), \ldots,\alpha_k(w))\,,
\end{align*}
where  $\alpha_i(w)=i-w(i)$. 

\begin{remark}\label{rem:alpha}
For any $w\in S_k$, we have that $\sum^k_{i=1}\alpha_i(w)=\sum_{i=1}^k(i-w(i))=0$. This implies that for any $t=(t_k,t_{k-1},\ldots,t_1)\in  \textbf{Z}_{>0}^k$ and $w\in S_k$, 
    $$|t|=\sum^k_{i=1}t_{k-i+1}=\sum^k_{i=1}(t_{k-i+1}+\alpha_i(w))=|t+\alpha(w)|.$$
If $t=(t_k,t_{k-1},\ldots,t_1)\in \textbf{Z}_{>0}^k$ satisfies $t_k>t_{k-1}>\cdots>t_1$,  then 
    $$t_{k-i+1}\geq t_{k-i}+1\geq\cdots\geq t_1+(k-i)$$ 
 for all $1\leq i\leq k$. Since $\alpha_i(w)=i-w(i)\ge i-k$, we have that 
    $$t_{k-i+1} + \alpha_i(w)\geq t_1+(k-i)+i-k\geq t_1>0$$ 
 for all $1\leq i\leq k$. This shows that $t+\alpha(w)$ is a composition of $|t|$. 
\end{remark}
\begin{definition} Given $t=(t_k,t_{k-1},\ldots,t_1)\in\textbf{Z}_{>0}^k$  with $t_k>t_{k-1}>\cdots>t_1>0$, let
\Squaresize=12pt
  $$F_{\tiny{\ruban{t_k\\
                 \vdots\\
                  t_1\\}}\quad }=\det(E_{t_{k-j+1}+(j-i)}) =\sum_{w\in S_k}(-1)^{l(w)}E_{t+\alpha(w)}.$$ 
Here $l(w)=Card\{(i,j)|i<j, w(i)>w(j)\}$. {}From Remark~\ref{rem:alpha} the operator $F$ is well defined and  takes a homogeneous polynomial of bidegree $(r,s)$ to a homogeneous polynomial of bidegree $(r-t_1-\cdots-t_k,s+k)$. 
\end{definition}
\begin{lemma}
Given $t=(t_k,t_{k-1},\ldots,t_1)\in\textbf{Z}_{>0}^k$  with $t_k>t_{k-1}>\cdots>t_1$, if for some $i$ we have $t_i=t_{i-1}+1$, then $F_{t }=0$.
\end{lemma}
\begin{proof}
If $t_i=t_{i-1}+1$ for some $1<i\leq k$, then the $(k-i+1)^{th}$ and $(k-i+2)^{th}$ columns in the determinant of $F$ are the same. 
\end{proof}

Therefore, we shall assume that $t_i\geq t_{i-1}+2$ in the definition of $F$-operators.
The following lemma is useful for our purpose.
\begin{lemma}\label{lem:Eaf}
Let $L=[(p_1,q_1),(p_2,q_2),\ldots,(p_n,q_n)]$ be a lattice diagram. We have
   $$E_j\Delta_L =\sum^n_{i=1}\Delta_{E_j^i(L)},$$
where 
   \begin{equation}\label{eq:Eij}
    E_j^i(L)=[(p_1,q_1),\ldots,(p_i-j,q_i+1),\ldots,(p_n,q_n)].
   \end{equation}
\end{lemma}
\begin{proof}
Recall that the determinant $\Delta_L = c\cdot \hbox{Alt}(x_1^{p_1}y_1^{q_1}\cdots x_n^{p_n}y_n^{q_n})$ where $c$ is a constant and Alt denotes the alternating sum over the symmetric group. For any symmetric operator $\Psi$, we have that $\hbox{Alt}\circ\Psi =\Psi\circ\hbox{Alt}$. The lemma follows using $\Psi=E_j$.
\end{proof}

We can also generalize this result to the case where there are several $E_i$'s acting consecutively on $\Delta_L$.
\begin{lemma}\label{lem:Eij}
Let $L=[(p_1,q_1),(p_2,q_2),\ldots,(p_n,q_n)]$ and $a=(a_k,a_{k-1},\ldots,a_1)$ a composition. We have:
$$E_a\Delta_L=E_{a_k}E_{a_{k-1}}\cdots E_{a_1}\Delta_L=\sum_{f:\{1,\ldots,k\}\rightarrow\{1,\ldots,n\}}\Delta_{E_a^f(L)}\,,$$
where 
$E_a^f(L)=E_{a_k}^{f(k)}E_{a_{k-1}}^{f(k-1)}\cdots E_{a_1}^{f(1)}(L)$ as in  Eq~$(\ref{eq:Eij})$.
\end{lemma}

Combining the definition of the operator $F$ with Lemma~\ref{lem:Eij} gives 
  $$F_{\tiny{\ruban{t_k\\
                 \vdots\\
                  t_1\\}}\quad }\Delta_L=\sum_{f:\{1,\ldots,k\}\rightarrow\{1,\ldots,n\} \atop w\in S_k}(-1)^{l(w)}\Delta_{E^f_{t+\alpha(w)}(L)}.$$

As the following theorem shows, many terms in this sum cancel.
\begin{theorem}\label{thm:Fop}
For $t=(t_k,t_{k-1},\ldots,t_1)\in\textbf{Z}_{>0}^k$ where $t_{i}\geq t_{i-1}+2$ for all $2\leq i\leq k$, and $L=[(p_1,q_1),(p_2,q_2),\ldots,(p_n,q_n)]$, we have
\begin{equation}\label{eq:fop}
F_{\tiny{\ruban{t_k\\
                 \vdots\\
                  t_1\\}}\quad }\Delta_L=\sum_{{f:\{1,\ldots,k\}\rightarrow\{1,\ldots,n\} \ injective\atop w\in S_k}}(-1)^{l(w)}\Delta_{E^f_{t+\alpha(w)}(L)}.
\end{equation}
For $f$ injective, we can explicitly describe $E^f_{t+\alpha(w)}(L)=[(p'_1,q'_1),(p'_2,q'_2),\ldots,(p'_n,q'_n)]$
\begin{equation*}
(p_s',q_s')=\left\{
\begin{array}{rl}
(p_s-t_i-\alpha_{k-i+1}(w),q_s+1), &\hbox{if } s=f(i),\\
(p_s,q_s), & \hbox{otherwise.}\\
\end{array} \right.
\end{equation*}
\end{theorem}

\begin{proof}
The left-hand side of equality~(\ref{eq:fop}) can be written into sums as follows:
\begin{align*}
F_{\tiny{\ruban{t_k\\
               \vdots\\
                t_1\\ }}\quad }\Delta_L=&\sum_{f:\{1,\ldots,k\}\rightarrow\{1,\ldots,n\}\atop w\in S_k}(-1)^{l(w)}\Delta_{E^f_{t+\alpha(w)}(L)}\\
                   =&\sum_{{f:\{1,\ldots,k\}\rightarrow\{1,\ldots,n\} \atop non\!-\!injective;}\atop w\in S_k}(-1)^{l(w)}\Delta_{E^f_{t+\alpha(w)}(L)}\\
                   &\qquad\qquad\hfill +\sum_{f:\{1,\ldots,k\}\rightarrow\{1,\ldots,n\} \ injective \atop w\in S_k}(-1)^{l(w)}\Delta_{E^f_{t+\alpha(w)}(L)}\, .
\end{align*}
By definition, $t+\alpha(w)=(t_k+1-w(1),t_{k-1}+2-w(2),\ldots,t_1+k-w(k))$ and thus $E^f_{t+\alpha(w)}(L)=E^{f(k)}_{t_k+1-w(1)}E^{f(k-1)}_{t_{k-1}+2-w(2)}\ldots E^{f(1)}_{t_{1}+k-w(k)}(L)$. If $f$ is non-injective, we can always find a pair $(i,j)$ such that $f(i)=f(j)$ where $k\geq i>j\geq 1$. The operator $E^f_{t+\alpha(w)}(L)$ related to such an $f$ is
\smallskip
\begin{equation}\label{eq:Fij}
E^{f(k)}_{t_k+1-w(1)}\ldots E^{f(i)}_{t_{i}+(k-i+1)-w(k-i+1)}\ldots E^{f(j)}_{t_{j}+(k-j+1)-w(k-j+1)}\ldots E^{f(1)}_{t_{1}+k-w(k)}(L).
\end{equation}
\smallskip
Let $\bar{w}=w(k-i+1,k-j+1)$, and notice that $\alpha(\bar{w})=(1-w(1),\ldots,k-i+1-w(k-j+1),\ldots,k-j+1-w(k-i+1),\ldots,k-w(k)).$ We then have that $E^f_{t+\alpha(\bar{w})}(L)$ is equal to
\smallskip
\begin{equation}\label{eq:Fji}
E^{f(k)}_{t_k+1-w(1)}\ldots E^{f(i)}_{t_{i}+(k-i+1)-w(k-j+1)}\ldots E^{f(j)}_{t_{j}+(k-j+1)-w(k-i+1)}\ldots E^{f(1)}_{t_{1}+k-w(k)}(L).
\end{equation}
\smallskip
{}From Equations~(\ref{eq:Fij}) and~(\ref{eq:Fji}), we can see that the only difference between $E^f_{t+\alpha(w)}(L)$ and $E^f_{t+\alpha(\bar{w})}(L)$ are the $E$-operators related to $i$ and $j$. Since $f(i)=f(j)$, we have
\begin{align*} 
  E^{f(i)}_{t_{i}+(k-i+1)-w(k-i+1)} E^{f(j)}_{t_{j}+(k-j+1)-w(k-j+1)}(L)&\\
=E&^{f(i)}_{t_{i}+(k-i+1)-w(k-j+1)} E^{f(j)}_{t_{j}+(k-j+1)-w(k-i+1)}(L).
\end{align*}
Indeed, this only changes the cell $(p_{f(i)},q_{f(i)})$ in $L$ into  $$(p_{f(i)}-t_i-t_j-(k-i+1)-(k-j+1)+w(k-i+1)+w(k-j+1),q_{f(i)}+2).$$ Since the $E$-operators commute, we conclude that $E^f_{t+\alpha(w)}(L)=E^f_{t+\alpha(\bar{w})}(L)$. 

For any $f$ non-injective, we pick the lexicographically unique pair $(i_f,j_f)$ such that $f(i_f)=f(j_f)$ and $ i_f>j_f\geq 1$.
We have
\begin{align*}
\sum_{{f:\{1,\ldots,k\}\rightarrow\{1,\ldots,n\}\atop \ non\!-\! injective;}\atop w\in S_k}& (-1)^{l(w)}\Delta_{E^f_{t+\alpha(w)}(L)}\\
=&\sum_{{f:\{1,\ldots,k\}\rightarrow\{1,\ldots,n\} \atop non\!-\!  injective;}\atop w\in S_k;\  l(w) \ even}
  (-1)^{l(w)}\big(  \Delta_{E^f_{t+\alpha(w)}(L)}-\Delta_{E^f_{t+\alpha(w(k-i_f+1,k-j_f+1))}(L)}\big) \\
=&\quad  0\,.
\end{align*}
This implies the theorem.
\end{proof}

\begin{remark}\label{rem:lead}
For $L=L(D)$ a sorted lattice diagram, let $f^o\colon\{1,\ldots,k\}\rightarrow\{1,\ldots,n\}$ be defined by $f^o(i)= i$.  We have 
\begin{align*}  
F_{t}\Delta_L =&\Delta_{E^{f^o}_{\,t}(L)}+ \sum_{(f,w)\not=(f^o,id)\atop f\ injective,\ w\in S_k}(-1)^{l(w)}\Delta_{E^f_{t+\alpha(w)}(L)}\\
              =&\Delta_{E^{f^o}_{\,t}(L)}+(lower\ terms).
\end{align*}
In particular, if $\Delta_{E^{f^o}_{\,t}(L)}\neq 0$,  then it is the leading diagram of $F_{t }\Delta_L$. 
To see this,  we show the following claim. For any $(f,w)\not=(f^o,\hbox{id})$, we have that
\begin{equation}\label{eq:LT}
E^f_{t+\alpha(w)}(L)\prec E^{f^o}_{\,t}(L)
\end{equation}
in the order defined in Eq~(\ref{eq:lexiorder}). 
Also if $L_1\prec L_2$, then for any $f$ we have
\begin{equation}\label{eq:LLT}
E_k^f(L_1)\prec E_k^f(L_2).
\end{equation} 
First consider the case where $w\not=\hbox{id}$. Let $s$ be the largest integer such that $\alpha_s(w)\neq 0$. We must have that $\alpha_s(w)=s-w(s)>0$ since
$\alpha_{s+1}(w)=\cdots=\alpha_k(w)=0$. We have $t_{k-i+1}+\alpha_i(w)=t_{k-i+1}$ for all $s+1\leq i\leq k$, and $t_{k-s+1}+\alpha_s(w)>t_{k-s+1}$. This gives
$$E^{f^o}_{\,t}(L)\succ E^{f^o}_{t+\alpha(w)}(L).$$
In the case where $w=\hbox{id}$, we have  $\alpha(\hbox{id})=(0,\ldots,0)$. In particular, for $f\ne f^0$,
$$E^{f^o}_{t+\alpha(w)}(L)\succ E^f_{t+\alpha(w)}(L).$$
The Eq~(\ref{eq:LT}) follows by transitivity. Eq~(\ref{eq:LLT}) is clear.
\end{remark}

\begin{example}
Let $t=(3,1)$ and $L=[(0,0),(1,0),(2,0),(3,0),(4,0)]$. Applying $F_t$ we get
\begin{align*}
F_{\tiny{\ruban{3\\
                1\\ }}\quad }\Delta_L =& \Delta_{[(0, 0),(1, 0),(2, 0), (0, 1),(3, 1)]}-3\Delta_{[(0, 0),(1, 0),(2, 0),(1, 1),(2, 1)]}\\
                                      &-3\Delta_{[(0, 0),(1, 0),(4, 0),(0, 1),(1, 1)]}+\Delta_{[(0, 0),(2, 0),(3, 0),(0, 1),(1, 1)]}\\
                                      &+2\Delta_{[(0, 0),(1, 0),(3, 0),(0, 1),(2, 1)]}
\end{align*}
Clearly $\Delta_{[(0, 0),(1, 0),(2, 0), (0, 1),(3, 1)]}$ is the leading diagram and it is obtained by 
\break
$E^2_3E^1_1([(0,0),(1,0),(2,0),(3,0),(4,0)])$.
\end{example}

\section{F-Operators on the space of alternating diagonal harmonics}

In this section we define the operator $F_T$ associated with a column-strict Young tableau $T$. We also construct an explicit basis of $A_n^{k,l}$ for $l\le 2$. This allows us to
better understand for which $T$ the elements $F_T\Delta_n$ are linearly independent.

Given a partition $\mu=(\mu_1,\mu_2,\ldots,\mu_l)$, let $D_{\mu}=\{(i,j)\in\textbf{Z}_{>0}^2: 1\leq j\leq\mu_i, \forall 1\leq i\leq l\}$. A {\sl tableau} $T$ is a function $T\colon D_{\mu}\rightarrow \textbf{Z}_{>0}$. Here the partition $\mu$ is called the {\sl shape} of $T$. By convention, set $T(i,j)=\infty$ if $(i,j)\not\in D_{\mu}$. The {\sl  transpose} of $\mu$ is the partition $\mu^t$ where $D_{\mu^t}$ is the transpose of $D_{\mu}$.  A {\sl column-strict tableau} is a tableau $T$ such that for all $(i,j)\in D_{\mu}$, we have $T(i,j)<T(i+1,j)$ and $T(i,j)\leq T(i,j+1)$. We usually let $t_{i,j}(T)=T(i,j)$ or just $t_{i,j}$ if it is unambiguous. We graphically put $t_{i,j}$ in a box at position $(i,j)$ in $D_{\mu}$. For example,
\Squaresize=12pt
$$T= \lower 10pt\hbox{\Young{\Carre{4}&\Carre{5}\cr
           \Carre{2}&\Carre{3}&\Carre{3}&\Carre{4}\cr
           \Carre{1}&\Carre{1}&\Carre{2}&\Carre{2}\cr
           }\ ,}$$  
is column-strict and $t_{2,3}=3$. Given a column-strict tableau $T$ we denote its shape by $\mu(T)=(\mu_1(T),\mu_2(T),\ldots,\mu_l(T))$ and row sum by $s(T)=(s_1(T),s_2(T),\ldots,s_l(T))$, where $s_i(T)=t_{i,1}+\cdots+t_{i,\mu_i}$. We may simply use $\mu=(\mu_1,\ldots,\mu_l)$  and $s=(s_1,\ldots,s_l)$ if the tableau is clear. We see that $(s_1,s_2,\ldots,s_l)$ forms a composition of $|s|=s_1+s_2+\cdots +s_l$, which is the sum of all numbers in the tableau $T$.  Let $T_j$ denote the $j^{th}$ column in the tableau $T$ for all $1\leq j\leq\mu_1$. We define the $F$-operator corresponding to $T$ by:
\begin{equation}\label{eq:Fop}
F_T:=F_{T_{\mu_1}}F_{T_{\mu_1-1}}\cdots F_{T_2}F_{T_1}.
\end{equation}
Since $T$ is column-strict, $F_T$ is well defined. 
Now we use it to generate sets of polynomials which are linearly independent.

\begin{theorem} \label{thm:l1}
For $1\le k\le n-1$, let $\mathcal{T}_k=\{\,{\tiny{\ruban{k\cr}}}\,\,\}$.
The set
$$B_{\mathcal{T}_k}=\{F_{\tiny{\ruban{k\cr}}}\,\Delta_n\}$$
is a basis of the space $A_n^{k,1}$ and $\dim A_n^{k,1}=1$. For all other values of $k$, $\dim A_n^{k,1}=0$.
\end{theorem}

\begin{proof}
Notice that the leading term of  $F_k\Delta_n$ is $\Delta_{[(0,0),(1,0),\ldots,(n-1-k,1)]}\neq 0$ for $1\leq k\leq n-1$.
It is clear that the set $B_{\mathcal{T}_k}=\{F_{\tiny{\ruban{k\cr}}}\,\Delta_n\}\ne\{0\}$ is linearly independent.  
Combining Eq~(\ref{eq:qtcatalan}) and Eq~(\ref{eq:catalan}), we see that $\dim A_n^{k,1}$ is given by the number of Dyck paths with bounce equal to $1$ and coarea equal to $k$. We then claim that  $\dim A_n^{k,1}=1$ for $1\le k\le n-1$ and zero otherwise.
{}This is because the only way to get $b(c)=1$ is if the Dyck sequence of $c$ is of the form 
    $$(0,1,2,\ldots,n-k-1,n-k-1,\ldots,n-3,n-2)$$
for $1\le k\le n-1$. 
\end{proof}

When $j\geq i+2$, we have $F_{\tiny{\ruban{j\cr i\cr}}}=E_{j ,i}-E_{j-1,i+1}$.
For $j=i$ or $i+1$, we have
$F_{\tiny{\ruban{i & j\cr}}}=F_{\tiny{\ruban{i\cr}}}F_{\tiny{\ruban{j\cr}}}=E_iE_j.$

\begin{theorem}\label{thm:l2}
For $2\le k\le 2n-2$, the set of polynomials $\{F_T\Delta_n\}$ where 
\Squaresize=16pt
\begin{equation*}
T =\left\{\begin{array}{ll}
\Young{\Carre{i}&\Carre{i}\cr} &\ \  i\leq n-2, \\
\Young{\Carre{i}&\Carre{i+1}\cr} & \ \ i\leq n-3, \\
\Young{\Carre{j}\cr\Carre{i}\cr} & \ 
\raise 15pt \hbox{$\begin{array}{l} i+2\leq j\leq n-2 \\ 1\leq i\leq n-4, \end{array}$} 
\end{array} \right.
\end{equation*} 
and $|s(T)|=k$ forms a basis of $A_n^{k,2}$. For all other values of $k$, $\dim A_n^{k,2}=0$.
\end{theorem}

\begin{proof}
We first show linear independence. When $2i\leq n-1$, the leading diagram of $E_iE_i\Delta_n$ is 
  $$\Delta_{[(0,0),(1,0),\ldots ,(n-2,0),(n-1-2i,2)]}\neq 0.$$
When $2i+1\leq n-1$, the leading diagram of $E_iE_{i+1}\Delta_n$ is 
  $$\Delta_{[(0,0),(1,0),\ldots ,(n-2,0),(n-2i-2,2)]}\neq 0.$$ 
When $2i\geq n$ and $ i\leq n-2$,
the leading diagram of $E_iE_i\Delta_n$ is 
  $$\Delta_{[(0,0),(1,0),\ldots,(n-3,0),(n-2-i,1),(n-1-i,1)]}\neq 0$$
 and when $2i\geq n$ and $ i\leq n-3$,
the leading diagram of $E_iE_{i+1}\Delta_n$ is\break 
  $$\Delta_{[(0,0),(1,0),\ldots,(n-3,0),(n-3-i,1),(n-1-i,1)]}\neq 0.$$
Now when $ i+2\leq j\leq n-2$ and $1\leq i\leq n-4$, using Remark~\ref{rem:lead}, the leading  diagram of  $F_{\tiny{\ruban{j\cr i\cr}}}\,\Delta_n$ is 
\vskip-14pt
  $$\Delta_{[(0,0),(1,0),\ldots,(n-3,0),(n-2-j,1),(n-1-i,1)]}\neq 0.$$
In each of the above cases, $F_T\Delta_n$  has a different leading diagram for different $T$, which implies that $\{F_T\Delta_n\}$ is linearly independent.

The only way to get $b(c)=2$ in Eq~(\ref{eq:qtcatalan}) is if the Dyck sequence of $c$ is of the form 
 $(g_0,g_1,\ldots,g_{n-1})$ where $g_{n-1}=n-3$ and $g_0=0$. {}As in Remark~\ref{rem:lambda} we consider the partition $\lambda=\mu^t$.
 The restriction on the Dyck sequence gives us that $\lambda$ has exactly two non-zero parts and the largest part is less than $n-2$.
 That is $\lambda=(j,i)$,\quad $j\le n-2$ and $i+j=k$.  In particular,  the coarea of such a path has to be  $2\le k\le 2n-2$.  The case $j=i$ or $i+1$ corresponds  to the tableaux $T$ of shape $(2)$ and when $j\ge i+2$ it corresponds to the tableaux $T$ of shape $(1,1)$. The dimension of $A_n^{k,2}$ has exactly the desired cardinality. For all other values of $k$, $\dim A_n^{k,2}=0$.
\end{proof}

\section{Definition and some properties of framed tableaux}\label{sec:NZT}

In order to generalize the results of the previous section, we need to study a special kind of Young tableau.
In Theorem~\ref{thm:l1} and~\ref{thm:l2}, a  basis of $A_n^{k,l}$ for $l\le 2$ is obtained from a set of the form $\{F_T\Delta_n\}$ where $T$ are well chosen. It suggests  that for $T$ with small shape, the rows must weakly increase with small differences and columns should have jump greater than or equal to 2. In the light of this observation, we introduce a new kind of tableau which  allow us to generalize the result for larger $l$.
\begin{definition}\label{def:framcond}
Given $\mu=(\mu_1,\mu_2,\ldots,\mu_l)$ and $s=(s_1,s_2,\ldots,s_l)$, we say that $(\mu,s)$ satisfies the {\sl framing condition} if \\
1. $s_i\geq (2i-1)\mu_i$ and\\
2. $s_{i+1}\geq s_i+2\mu_i$ for all $1\leq i\leq l-1$ such that $\mu_{i+1}=\mu_i$.\\
\end{definition}
Our goal is to build a new kind of tableau $T_{(\mu,s)}$ with shape $\mu$ and row sum $s$. The following definition is useful for our algorithms.
\begin{definition}\rm
Given an integer $c\in\textbf{Z}_{>0}$, there is a unique way to decompose it into $m$ positive integers $c_1\leq c_2\leq\cdots\leq c_m$ such that $c=c_1+\cdots+c_m$ and $0\leq c_j-c_i\leq 1$ for all $1\leq i<j\leq m$. We say that $\hbox{B-comp}(c,m)=(c_1,c_1,\ldots,c_m)$ is the  {\sl balance composition} of $c$.
\end{definition}
Given $(\mu,s)$ satisfying the framing condition in Definition~\ref{def:framcond}, we give  a procedure that constructs a unique column-strict tableau of shape $\mu$ and row sum $s$. We call the procedure {\sl framing} and the resulting tableau a {\sl framed tableau}.  By convention let $\mu_{l+1}=0$.

\medskip
{\parskip=.05cm
\noindent
\textbf{Fram}$\big(\mu=(\mu_1,\mu_2,\ldots,\mu_l),s=(s_1,s_2,\ldots,s_l)\big)$ 

$(t_{l,1},t_{l,2},\ldots,t_{l,\mu_l}):=\hbox{B-comp}(s_l,\mu_l)$

\textbf{For} $i=l-1$ \textbf{Downto} 1 \textbf{Do}

\hskip .5cm $k:=l;\  a:=s_i;\  b:=\mu_i$

\hskip .5cm \textbf{While} $k\geq i$ \textbf{Do}

\hskip 1cm $(r_{i,\mu_{k+1}+1},\ldots,r_{i,\mu_i})=\hbox{B-comp}(a,b)$

\hskip 1cm \textbf{If} $r_{i,j}\leq t_{i+1,j}-2$, \ $\forall \mu_{k+1}+1\leq j\leq\mu_k$

\hskip 1.5cm \textbf{Then} $t_{i,j}:=r_{i,j}$, \  $\forall \mu_{k+1}+1\leq j\leq\mu_k$

\hskip 1.5cm \textbf{Else} $t_{i,j}:=t_{i+1,j}-2$, \  $\forall \mu_{k+1}+1\leq j\leq\mu_k$

\hskip 1cm $a:=a-(t_{i,\mu_{k+1}+1}+\cdots+t_{i,\mu_k})$; \  $b:=b-(\mu_k-\mu_{k+1})$; \ $k:=k-1$; 

\textbf{Output} $T=[t_{i,j}]$.}

\medskip\noindent
We write  $T=\hbox{Fram}(\mu,s)$. 

\begin{remark}\rm
The framing procedure is well defined and gives a unique framed tableau $\hbox{Fram}(\mu,s)$ for each $(\mu,s)$ satisfying the framing condition. The top row is clearly unique. Suppose we perform the procedure properly and uniquely up to row $i+1$.  For row $i$, if $\mu_i>\mu_{i+1}$, then the procedure works well. If $\mu_i=\mu_{i+1}$, then the framing condition gives that $s_i+2\mu_i\leq s_{i+1}$, which guarantees that the procedure produces a unique tableau.
\end{remark}
\begin{proposition}\label{prop:inj}
The framing procedure is an injection from the $(\mu,s)$ which satisfies the framing condition for column-strict tableaux. We call {\sl framed tableaux} the subset of tableaux in the image of {\rm Fram}.
\end{proposition}
  
\begin{example}  
For a given $s=(22,18,24,14)$ and $\mu=(8,5,4,2)$, we construct the corresponding framed tableau $\hbox{Fram}(\mu,s)$ with the above procedure:
\vskip-10pt
\Squaresize=12pt
$$\Young{\Carre{7}&\Carre{7}\cr
          }\ \raise 6pt \hbox{$\rightarrow$}\ 
         \Young{\Carre{7}&\Carre{7}\cr
           \Carre{5}&\Carre{5}&\Carre{7}&\Carre{7}\cr 
           }\ \raise 6pt \hbox{$\rightarrow$}\ 
         \Young{\Carre{7}&\Carre{7}\cr
           \Carre{5}&\Carre{5}&\Carre{7}&\Carre{7}\cr 
           \Carre{3}&\Carre{3}&\Carre{4}&\Carre{4}&\Carre{4}\cr 
           }\ \raise 6pt \hbox{$\rightarrow$}\ 
         \Young{\Carre{7}&\Carre{7}\cr
           \Carre{5}&\Carre{5}&\Carre{7}&\Carre{7}\cr 
           \Carre{3}&\Carre{3}&\Carre{4}&\Carre{4}&\Carre{4}\cr 
           \Carre{1}&\Carre{1}&\Carre{2}&\Carre{2}&\Carre{2}&\Carre{4}&\Carre{5}&\Carre{5}\cr 
           }\ . $$  
\end{example}
We have the following properties for framed tableaux.
\begin{lemma}\label{lem:framcar}
A framed tableau T of shape $\mu$, is a column-strict Young tableau of shape $\mu$ satisfying the following properties:

1. Any two numbers in the same column differ by at least $2$. 

2. For any $a\le b$ in the same row of $T$ we have $b-a\le 1$, unless there is a number $d$ above $a$ such that $d=a+2$.
\end{lemma}

To illustrate this, consider the following picture:
$$
\begin{array}{ccc}
 d & & \\
 a &\cdots & b 
 \end{array} \hskip -50pt
  \raise -50pt\hbox{ \begin{picture}(80,80)
       \put(-21,30){\line(0,1){50}}
       \put(-21,80){\line(1,0){51}}
       \put(30,80){\line(0,-1){15}}
       \put(30,65){\line(1,0){30}}
       \put(60,65){\line(0,-1){20}}
       \put(60,45){\line(1,0){46}}
       \put(106,45){\line(0,-1){15}}
       \put(106,30){\line(-1,0){127}}
      \end{picture}}\qquad\quad\lower 18pt\hbox{.} 
$$ 
\vskip-24pt
\noindent
Normally, $a$ and $b$ need to satisfy the condition $b-a\leq 1$. However if $d=a+2$, then there is no restriction on $b-a$. 
To prove Lemma~\ref{lem:framcar}, we need the following auxiliary  result.

\begin{lemma}\label{lem:twoseq}
Suppose that we have two sequences of integers $c_1\leq c_2\leq\cdots\leq c_{n}$ and $d_1\leq d_2\leq\cdots\leq d_{n+m}$ satisfying $c_j-c_i\leq 1$, for all $1\leq i<j\leq n$ and $d_j-d_i\leq 1$, for all $1 \leq i<j\leq n+m$. If there is a $k$ such that $c_k-d_k\leq 1$, then $c_j-d_j\leq 2$, for all $1\leq j\leq n$.  
\end{lemma}
\begin{proof}
For $1\leq j<k$, we have  $c_j=c_k$ or $c_k-1$ and, $d_j=d_k$ or $d_k-1$. Hence $c_j-d_j\leq c_k-(d_k-1)\leq 2$. For $k\leq j<n$,we have  $c_j=c_k$ or $c_k+1$ and, $d_j=d_k$ or $d_k+1$. Hence $c_j-d_j\leq c_k+1-d_k\leq 2$. 
\end{proof}

\begin{proof}[Proof of Lemma~\ref{lem:framcar}]
It is clear that  the row $l$ of $T_{(s,\mu)}$ given by  $\hbox{B-comp}(s_l,\mu_l)=(t_{l,1},t_{l,2},\ldots,t_{l,\mu_l})$ satisfies Properties 1 and 2. By induction, suppose that up to row $i+1$, Properties 1 and 2 are satisfied. Moreover, we assume (by induction) that for $i+1\leq k'\leq l$, we have

3. $t_{k',j_2}-t_{k',j_1}\leq 1$, for all $\mu_{k'+1}+1\leq j_1<j_2\leq\mu_{k'}$. \\
Recall here that we let $\mu_{l+1}=0$. For row  $i$, we consider the while loop of the framing procedure.
The properties 1,2 and 3, certainly hold whenever $t_{i+1,j}\geq r_{i,j}+2$, for $ 1\leq j\leq \mu_{i+1}$. If at one point, for $i\leq k\leq l$, there is $\mu_{k+1}+1\leq j_0\leq\mu_k$ such that $t_{i+1,j_0}-r_{i,j_0}\leq 1$, then by Lemma~\ref{lem:twoseq} we have $t_{i+1,j}-r_{i,j}\leq 2$ for all $\mu_{k+1}+1\leq j\leq\mu_k$. The framing procedure sets all $t_{i,j}:=t_{i+1,j}-2$, for $\mu_{k+1}+1\leq j\leq\mu_k$.  When we compare $a:=a-(t_{i,\mu_{k+1}+1}+\cdots+t_{i,\mu_k})$ with $a':=a-(r_{i,\mu_{k+1}+1}+\cdots+r_{i,\mu_k})$, we obtain $a'\leq a$. Hence $\hbox{B-comp}(a',b)\leq \hbox{B-comp}(a,b)$ component-wise. This implies that the row is weakly increasing. Properties 1, 2 and 3 also hold in this case.
\end{proof}

\begin{example} The following are framed tableaux of shape $(3,2,1)$:
\Squaresize=12pt
$$\Young{\Carre{6}\cr
           \Carre{3}&\Carre{4}\cr
           \Carre{1}&\Carre{2}&\Carre{2}\cr
           }\, ,\ 
          \Young{\Carre{6}\cr
           \Carre{3}&\Carre{3}\cr
           \Carre{1}&\Carre{1}&\Carre{4}\cr
           }\, ,\ 
          \Young{\Carre{6}\cr
           \Carre{4}&\Carre{4}\cr
           \Carre{1}&\Carre{1}&\Carre{2}\cr
           }\, ,\ 
          \Young{\Carre{5}\cr
           \Carre{3}&\Carre{6}\cr
           \Carre{1}&\Carre{1}&\Carre{2}\cr
           }\, ,\ 
          \Young{\Carre{5}\cr
           \Carre{3}&\Carre{7}\cr
           \Carre{1}&\Carre{4}&\Carre{5}\cr
           }\ .$$

The following is not a framed tableau: 
\Squaresize=12pt
$\Young{\Carre{4}&\Carre{5}\cr
           \Carre{1}&\Carre{2}&\Carre{6}\cr
           }\ ,$
since the difference between $1$ and $6$ is greater than $1$, but the number above $1$ is $4$, which is not exactly $2$ more than $1$.   
\end{example}

 {}From the definition and properties of framed tableaux the following lemmas can be obtained immediately:
\begin{lemma}
Suppose $T$ is a framed tableau. If we add or subtract some constant $k$ from each number in $T$ and if all entries remain positive, then we get a framed tableau and we denote this by $T\pm k$.
\end{lemma}

\begin{lemma}\label{lem:TopT}
Suppose $T$ is a framed tableau. If we delete the bottom row in $T$, then the remaining tableau is still a framed tableau.
\end{lemma}

\begin{lemma}
For a framed tableau $T$, suppose the $i^{th}$ and $j^{th}$ columns $(i<j)$ have the same height. We list the entries of each column, from bottom to top, as $\{a_1,a_2,\ldots,a_n\}$ and $\{b_1,b_2,\ldots,b_n\}$ respectively. Then $b_k-a_k\leq 1$, for all $1\leq k\leq n$.
\end{lemma}
\begin{proof}
Since the $i^{th}$ column and the $j^{th}$ column $(i<j)$ have the same height in the framed tableau $T$, by the framing procedure we know that there exists a $k$ such that $\mu_{k+1}+1\leq i<j\leq \mu_k$. We have $b_n-a_n\leq 1$ since both entries are parts of a balance composition. By induction, assume $b_s-a_s\leq 1$ for $1\le k< s\leq n$. For $k$, either the entries $a_k$ and $b_k$ are parts of a balance composition, or they satisfy $a_{k}=a_{k+1}-2$ and $b_{k}=b_{k+1}-2$. By the induction hypothesis we have in the latter case $b_{k}-a_{k}=b_{k+1}-a_{k+1}\leq 1$.
\end{proof}
\begin{remark}\label{rem:shape}
The contrapositive of this lemma gives the following: suppose columns $i$ and $j$ are $a_1,\ldots,a_n$ and $b_1,\ldots,b_m$ respectively, and there exists some $k$ such that $b_k-a_k\geq 2$, then $n>m$. That is column $i$ is strictly higher than column $j$. We use this to detect the structure of framed tableaux.
\end{remark}

\section{Insertion and Taquin}\label{sec:ITS}
The goal of this section is to describe the procedures that  give a bijection between partitions $\lambda\vdash n$ and framed tableaux with row sum $s=(s_1,\ldots, s_l)\models n$. For this purpose we need a procedure similar to Schensted's algorithm, with some additional straightening steps to get a framed tableau.

Given a framed tableau $T=[t_{i,j}]$ of shape $\mu$ and $0<x\leq t_{1,1}$, we define a procedure to insert $x$ into $T$ and denote the resulting framed tableau by 
$T\leftarrow x$. The algorithm is done in three steps. First we do an insertion that  gives an auxiliary tableau $Y$.  The tableau $Y$ determines a shape $\mu'=\mu(Y)$.
We use $Y$ in the second step to determine a row sum $s'$ such that $(\mu',s')$ satisfies the framing condition of Definition~\ref{def:framcond}.
Finally, $T\leftarrow x$ is given by $\hbox{Fram}(\mu',s')$. In our pseudocode, a loop of the form "{\bf For} ... {\bf To} ... {\bf While} $A$ {\bf Do} ..."
is a standard {\bf For} loop that stops as soon as $A$ is false.

Recall that  for a framed tableau  $T=[t_{i,j}]$ of shape $\mu=\mu(T)=(\mu_1,\ldots,\mu_l)$ we assume that $\mu_{l+1}=0$ and $t_{i,j}=\infty$ for $j>\mu_i$.

\medskip
{\parskip=.05cm
\noindent
$\mathbf{T\leftarrow x}$

\noindent {\sl Step 1: auxiliary insertion, getting $Y$ and $\mu'$}

$i:=1$; \ $x_0:=x$

\textbf{While} $t_{i,\mu_i}\ge x+2$  \textbf{do}

\hskip .5cm $j:=\hbox{Min}\{j: t_{i,j}\geq x+2\}$

\hskip .5cm $(t_{i,1},\ldots,t_{i,\mu_i}):=\hbox{Sort}(t_{i,1},\ldots,t_{i,j-1},x,t_{i,j+1},\ldots,t_{i,\mu_i})$

\hskip .5cm $x:=t_{i,j}$; \ $i:=i+1$

$(t_{i,1},\ldots,t_{i,\mu_i+1}):=\hbox{Sort}(t_{i,1},\ldots,t_{i,\mu_i},x)$

$Y:=[t_{i,j}]$; \  $\mu':=\mu(Y)$; \ $l':=\hbox{length}(\mu')$

\noindent {\sl Step 2: finding the new row sum $s'$}

$x:=x_0$; \ $(s_1,\ldots,s_{l'}):=s(Y)$

\textbf{For} $i=2$ \ \textbf{To} \ $l'$ \  \textbf{Do} \  $d_i:=0$

\textbf{For} $k=1$ \ \textbf{To} \ $l'-1$ \  \textbf{While} \ $t_{k,\mu'_k}\ge x+2$\  \textbf{Do} 

\hskip .5cm \textbf{For} $j=1$ \ \textbf{To} \ $\mu'_{k+1}$ \  \textbf{Do}

\hskip 1cm \textbf{If} $t_{k,j}=x$  \textbf{Then} $t_{k+1,j}:=x+2$

\hskip 1cm \textbf{If} $t_{k,\mu'_k}>x+2$ and $t_{k,j}=x+1$  \textbf{Then} $t_{k+1,j}:=x+3$

\hskip .5cm $\bar{s}_{k+1}:=t_{k+1,1}+t_{k+1,2}+\cdots+t_{k+1,\mu'_{k+1}}$

\hskip .5cm $d_{k+1}:=s_{k+1}-\bar{s}_{k+1}$

\hskip .5cm $x:=x+2$

$s':=(s_1+d_2,s_2+d_3-d_2,\ldots,s_{l'-1}+d_{l'}-d_{l'-1},s_{l'}-d_{l'})$

\noindent {\sl Step 3:}

\textbf{Output}  Fram$(\mu', s')$
}

\medskip
We show in Section~\ref{sec:bijec} that the $\mathbf{T\leftarrow x}$ algorithm is well defined for $0<x\leq t_{1,1}$ and produces a framed tableau. We give a short example to better demonstrate the steps.

\begin{example} Let $x=1$ for 
\Squaresize=12pt
$$T=\lower 4pt\hbox{$\Young{\Carre{4}\cr
           \Carre{2}&\Carre{5}&\Carre{6}&\Carre{6}\cr
           }$} \,\hbox{.  In Step 1, we get }\quad Y =\,\  \lower 4pt\hbox{$\Young{\Carre{4}&\Carre{5}\cr
           \Carre{1}&\Carre{2}&\Carre{6}&\Carre{6}\cr
           }$}\ .$$
Notice that the resultant tableau $Y$ from step $1$ may not be a framed tableau. We need to straighten $Y$ to get a framed tableau. We get $\mu'=(4,2)$ and $s(Y)=(15,9)$.
The second loop in Step 2 sets $d_2=1+1=2$ and in the end $s'=(15+2,9-2)=(17,7)$. The pair $(\mu',s')$ satisfies the framing condition so we can apply the framing procedure and get
\Squaresize=12pt
$$T\leftarrow 1 =\hbox{Fram}(17,7)=\ \lower10pt\hbox{\Young{\Carre{3}&\Carre{4}\cr
           \Carre{1}&\Carre{2}&\Carre{7}&\Carre{7}\cr}\ .} $$
 \end{example}

Most entries in $\mathbf{T\leftarrow x}$ might be different from those in $T$. But we remark that all the entries with value equal to $x$ or $x+1$ in the tableau $Y$ in Step 1  still remain unchanged in $T\leftarrow x$, and $x$ is the smallest entry in $\mathbf{T\leftarrow x}$. This fact is important and allows us to introduce an inverse procedure. This is done by playing a variation of Jeu de Taquin. Again this is done in three steps. We start with a framed tableau $T=[t_{i,j}]$ of shape $\mu=\mu(T)=(\mu_1,\ldots,\mu_l)$ and assume that $x=t_{1,1}$. We denote by $_xT$  the framed tableau we obtain by removing $x$ from $T$ with the following procedure:

{\parskip=.05cm
\noindent
{$\mathbf{_xT}$}

\noindent {\sl Step 1: Jeu de Taquin to get $Y$ and $\mu'$}

$i:=1$; \  $j:=1$

$x:=t_{1,1}$

\textbf{While} $t_{i+1,j}<\infty$ or $t_{i,j+1}<\infty$ \textbf{Do}

\hskip .5cm \textbf{If} $t_{i+1,j}\geq t_{i,j+1}+2$  \textbf{Then} $t_{i,j}:=t_{i,j+1}$; \  $j:=j+1$

\hskip .5cm \textbf{If} $t_{i+1,j}\leq t_{i,j+1}+1$ 

\hskip 1cm  \textbf{Then} 
 $t_{i,j}:=t_{i+1,j}$; \ $(t_{i,1},\ldots,t_{i,\mu_i}):=\hbox{Sort}(t_{i,1},\ldots,t_{i,\mu_i})$; \  $i:=i+1$

$t_{i,j}:=\infty$;

 $Y:= [t_{i,j}]$; \ $\mu':=\mu(Y)$; \  $l':=\hbox{length}(\mu')$

\noindent {\sl Step 2: row sum $s'$}

$(s_1,\ldots,s_{l'}):=s(Y)$

\textbf{For} $i=2$ \ \textbf{To} \ $l'$ \   \textbf{Do} \  $d_i:=0$

\textbf{For} $k=1$ \ \textbf{To} \ $l'-1$ \  \textbf{While} $t_{k,\mu'_k}\ge x+2$ and $t_{k,1}\le x+1$\ \textbf{Do} 

\hskip .5cm \textbf{For} $j=1$ \ \textbf{To} \ $\mu'_{k+1}$ \  \textbf{Do} 

\hskip 1cm \textbf{If}  $t_{k,j}=x$  \textbf{Then} $t_{k+1,j}:=x+2$

\hskip 1cm \textbf{If} $t_{k,\mu'_k}> x+2$ and  $t_{k,j}=x+1$  \textbf{Then} $t_{k+1,j}:=x+3$

\hskip .5cm $\bar{s}_{k+1}:=t_{k+1,1}+t_{k+1,2}+\cdots+t_{k+1,\mu'_{k+1}}$

 \hskip .5cm $d_{k+1}:=s_{k+1}-\bar{s}_{k+1}$

 \hskip .5cm $ x:=x+2$

$s':=(s_1+d_2,s_2+d_3-d_2,\ldots,s_{l'-1}+d_{l'}-d_{l'-1},s_{l'}-d_{l'})$

\noindent {\sl Step 3:}

\textbf{Output}  Fram$(\mu', s')$
}

\medskip  Again, we  show in the next section that this algorithm works and is well defined. We give here a short example to better show the steps.

\vskip -20pt
\begin{example} Given a framed tableau
$T=\lower 6pt\hbox{ \Young{\Carre{6}\cr
           \Carre{4}&\Carre{5}\cr
           \Carre{1}&\Carre{1}\cr
           }}\ ,$
we remove $x=1$ from $T$ in the first step of $_xT$. We use a dot to indicate the position of the cell as we perform the jeu de taquin.
\vskip -20pt
$$ \Young{\Carre{6}\cr
           \Carre{4}&\Carre{5}\cr
           \Carre{\bullet}&\Carre{1}\cr
           }\ \raise 6pt \hbox{$\rightarrow $}\    
 \Young{\Carre{6}\cr
           \Carre{4}&\Carre{5}\cr
           \Carre{1}&\Carre{\bullet}\cr
           }\ \raise 6pt \hbox{$\rightarrow $}\ 
 \Young{\Carre{6}\cr
           \Carre{4}&\Carre{\bullet}\cr
           \Carre{1}&\Carre{5}\cr
           }\ \raise 6pt \hbox{$\rightarrow $}\ 
 \Young{\Carre{6}\cr
           \Carre{4}\cr
           \Carre{1}&\Carre{5}\cr
           }\  \  \raise 6pt \hbox{=\ Y\ .}$$                                 
Again, $Y$ may not be a framed tableau. We have $\mu'=\mu(Y)=(2,1,1)$ and $s(Y)=(6,4,6)$. The second loop in Step 2 sets $d_2=1$ and $d_3=0$. In the end $s'=(6+1,4+0-1,6-0)=(7,3,6)$. The pair $(\mu',s')$ satisfies the framing condition so we can apply the framing procedure and get
\vskip -20pt
$$_xT=\hbox{Fram}(\mu',s')=\  \lower 15pt\hbox{\Young{\Carre{6}\cr
           \Carre{3}\cr
           \Carre{1}&\Carre{6}\cr
           }\,.}$$
 \end{example}
           
\section{1-1 Correspondence between partitions and framed tableaux}\label{sec:bijec}

In this section we construct a $1-1$ correspondence between partitions and framed tableaux.
Given a partition $(\lambda_1,\ldots,\lambda_k)\vdash l$, we get a framed tableau as follows:
\begin{equation}\label{eq:insersion}
\emptyset\leftarrow\lambda:=(\cdots((\emptyset\leftarrow \lambda_1)\leftarrow\lambda_2)\cdots\leftarrow\lambda_k).
\end{equation}
On the other hand, given a framed tableau $T$, we get a partition $\lambda(T)$ by recording the numbers removed each time with
\begin{equation}\label{eq:taquin}
_{x_k}(\cdots \,_{x_{2}}(\ _{x_1}T)\cdots )=\emptyset.
\end{equation}
Then $\lambda(T):=(x_k,\ldots,x_2,x_1)\vdash |s(T)|$. This is not the shape of $T$ that we denoted by $\mu(T)$. We prove here that these two maps are inverse to each other, and thus there is a bijection between partitions and framed tableaux. First, we give a lemma to reduce the number of cases we have to consider.
\begin{lemma}\label{lem:red}
Given $T=[t_{i,j}]$ and $0<x+1\leq t_{1,1}$, we have
$$T\leftarrow (x+1)=\big((T-x)\leftarrow 1\big)+x,$$
$$_{x+1}T=\big(\ _1(T-x)\big)+x.$$
\end{lemma}
\begin{proof}
The result of $T\leftarrow (x+1)$ is determined by the differences between $x$ and the $t_{i,j}$'s. It is clear that $\mu\big(T\leftarrow (x+1)\big)=\mu\big((T-x)\leftarrow 1\big)$. Furthermore, the $d_i$'s in the procedure $T\leftarrow (x+1)$ and $(T-x\big)\leftarrow 1$ are the same for all $2\leq i\leq l$. Thus we have that $s\big(T\leftarrow (x+1)\big)=s\big(\big((T-x)\leftarrow 1\big)+x\big)$. In both cases, we produce the same framed tableau. The argument for the second equality is similar.
\end{proof}

\begin{theorem}\label{thm:well}
Let $T=[t_{i,j}]$ and $0<x\leq t_{1,1}$.\hfill\break
\hbox{\rm (a)} The procedures $T\leftarrow x$ and $_xT$ are well defined and inverse to each other.\\
\hbox{\rm (b)} The maps defined by $(\ref{eq:insersion})$ and $(\ref{eq:taquin})$ give a bijection $\lambda\leftrightarrow T$ between $\lambda\vdash k$ with $l$ parts and the framed tableaux $T$ such that $\mu(T)\vdash l$ and $s(T)\models k$. 
\end{theorem}
\begin{proof}
Part (b) follows directly from Part (a). We show (a) case by case. {}From Lemma~\ref{lem:red} it is sufficient to consider only the cases of inserting and removing $1$ from any given framed tableau. 
Let $T\colon D_{\mu}\rightarrow  \textbf{Z}_{>0}$ be a framed tableau with shape $\mu$ and row sum $s=(s_1,s_2,\ldots,s_l)$:

$$
T=\begin{array}{ccc}
 \vdots & & \\
 t_{2,1} & t_{2,2} & \\
 t_{1,1} & t_{1,2} & \cdots  
 \end{array} \hskip -55pt
  \raise -50pt\hbox{ \begin{picture}(120,80)
        \put(-21,30){\line(0,1){60}}
       \put(-21,90){\line(1,0){51}}
       \put(30,90){\line(0,-1){25}}
       \put(30,65){\line(1,0){30}}
       \put(60,65){\line(0,-1){20}}
       \put(60,45){\line(1,0){46}}
       \put(106,45){\line(0,-1){15}}
       \put(106,30){\line(-1,0){127}}
      \end{picture}}\lower 20pt \hbox{\,.}
$$ \vskip -10pt
\noindent
{\bf Case 1.} Assume that $t_{1,1}\geq 3$ and let $m=\mu^t_2$. In Step 1 of $T\leftarrow 1$, we obtain
$$
Y=\begin{array}{lll} 
 t_{l,1} & & \\
\vdots & & \\ 
t_{m,1}& & \\ 
t_{m-1,1}& t_{m,2}& \\
 \vdots &\vdots & \\
 t_{1,1} & t_{2,2} & \\
 1 & t_{1,2} & \cdots  
 \end{array} \hskip -57pt
  \raise -65pt\hbox{ \begin{picture}(100,60)
       \put(-38,14){\line(0,1){92}}
       \put(-38,106){\line(1,0){30}}
       \put(-8,106){\line(0,-1){30}}
       \put(-8,76){\line(1,0){58}}
       \put(50,76){\line(0,-1){26}}
       \put(50,50){\line(1,0){40}}
       \put(90,50){\line(0,-1){36}}
       \put(90,14){\line(-1,0){128}}
      \end{picture}} \lower 50pt \hbox{\,.}
$$ \vskip -10pt\noindent
Let $[y_{i,j}]=Y$. We have that $s(Y)=(s_1-t_{1,1}+1,s_2-t_{2,1}+t_{1,1},\ldots,s_l-t_{l,1}+t_{l-1,1},t_{l,1})$ and  $\mu'=\mu(Y)=(\mu_1,\ldots,\mu_l,1)$. 
In Step 2 of $T\leftarrow 1$, as $k$ varies from row $1$ to $m$, we have $x=2k-1$. Since $t_{1,1}\ge 3$ we have $y_{k,\mu_k}\ge  t_{k,1}\geq t_{1,1}+2k-2\ge 2k+1= x+2$.
The entries in the first column are sequentially changed to $y_{k+1,1}:=2k+1$ since the entry in row $k$ is $y_{k,1}=2k-1=x$. 
No other entries are changed since for $j\ge 2$ we have $y_{k,j}\geq t_{k,1}\geq 2k+1>x+1$. 
The loop stops after $k=m$ since for row $m+1$, we have $x=2m+1$ and at that moment, $y_{m+1,\mu'_{m+1}}=2m+1=x\not\ge x+2$.
 We then have that
$d_2=t_{1,1}-3$, $\ldots$, $d_{m+1}=t_{m,1}-(2m+1)$, $d_{m+2}=\cdots=d_l=0$. Thus the new row sum is 
$s'=(s_1-2,s_2-2,\ldots,s_m-2,2m+1,t_{m+1,1},\ldots,t_{l,1})$. Since $(\mu,s)$ satisfies the framing condition of Definition~\ref{def:framcond} it is easy to check that $(\mu',s')$ also satisfies the framing condition. We can thus compute $\hbox{Fram}(\mu',s')$ in Step 3 of $T\leftarrow 1$. Since $\mu'_{m+1}=1$ and $s'_{m+1}=2m+1$, we must have that the first entry of each row $1\le k\le m+1$ of $\hbox{Fram}(\mu',s')$ is $2k-1$. 
We obtain
\vskip -10pt
$$
T'=T\leftarrow 1=\hbox{Fram}(\mu',s')=\begin{array}{lcc} 
 t_{l,1} & & \\
\vdots & & \\ 
\scriptstyle{2m+1}& & \\ 
\scriptstyle{2m-1}& t'_{m,2}& \\
 \vdots &\vdots & \\
3 & t'_{2,2} & \\
 1 & t'_{1,2} & \cdots  
 \end{array} \hskip -50pt
  \raise -65pt\hbox{ \begin{picture}(100,60)
       \put(-38,14){\line(0,1){92}}
       \put(-38,106){\line(1,0){30}}
       \put(-8,106){\line(0,-1){30}}
       \put(-8,76){\line(1,0){58}}
       \put(50,76){\line(0,-1){26}}
       \put(50,50){\line(1,0){40}}
       \put(90,50){\line(0,-1){36}}
       \put(90,14){\line(-1,0){128}}
      \end{picture}}  \lower 50pt \hbox{\,.}
$$ 
\vskip -10pt 
\noindent 
For $1\le k\le m$, we have $t'_{k,2}+\cdots+t'_{k,\mu_k}=s'_k-2k+1=s_k-2k-1\ge s_k-t_{k,1}=t_{k,2}+\cdots+t_{k,\mu_k}$.
This implies $t'_{k,j}\geq t_{k,j}$ for all $1\le k\le m$ and  $j\geq 2$. 
Now we want to show $_1T'=T$. In Step 1 of $_1T'$, we get
$$
Y=\ \begin{array}{lcc}
\raise 15pt \hbox{$ t_{l,1} $}& & \\
 \raise 8pt \hbox{\vdots }& & \\
5& t'_{2,2} & \\
3& t'_{1,2} & \cdots  
 \end{array}\hskip -40pt
  \raise -55pt\hbox{ \begin{picture}(95,108)
       \put(-38,14){\line(0,1){92}}
       \put(-38,106){\line(1,0){30}}
       \put(-8,106){\line(0,-1){30}}
       \put(-8,76){\line(1,0){58}}
       \put(50,76){\line(0,-1){26}}
       \put(50,50){\line(1,0){40}}
       \put(90,50){\line(0,-1){36}}
       \put(90,14){\line(-1,0){128}}
      \end{picture}} \lower 40pt \hbox{\,.}$$ 
 \vskip -10pt\noindent
We now get that $\mu'=\mu(Y)=\mu$ and $s(Y)=(s_1,\ldots,s_l)$. Let $[y_{i,j}]=Y$. Since $y_{1,1}=3\not\le 2=x+1$, we do not do any loops in Step 2.
Clearly $(\mu,s)$ satisfies the framing condition and
$_1T'=\hbox{Fram}(\mu,s)=T.$

\medskip
\noindent
{\bf Case 2.} Row $k=1$ of $T$ contains only $1$'s or $2$'s or both. Let

\begin{align*}
\raise 20pt\hbox{$T=\ \ $}\begin{array}{cccccc}
1 & \cdots & 1 & 2&\cdots & 2
 \end{array} \hskip -100pt
  \raise -32pt\hbox{ \begin{picture}(100,50)
        \put(-21,30){\line(0,1){60}}
       \put(-21,90){\line(1,0){51}}
       \put(30,90){\line(0,-1){25}}
       \put(30,65){\line(1,0){30}}
       \put(60,65){\line(0,-1){15}}
       \put(60,50){\line(1,0){35}}
       \put(95,50){\line(0,-1){20}}
       \put(95,30){\line(-1,0){116}}
      \end{picture}}\quad\quad
&\raise 20pt \hbox{and}
&\raise 20pt\hbox{$T'=\ $}\begin{array}{ccccccc}
1& 1 & \cdots & 1 & 2&\cdots & 2
 \end{array} \hskip -112pt
  \raise -32pt\hbox{ \begin{picture}(110,80)
        \put(-21,30){\line(0,1){60}}
       \put(-21,90){\line(1,0){51}}
       \put(30,90){\line(0,-1){25}}
       \put(30,65){\line(1,0){30}}
       \put(60,65){\line(0,-1){15}}
       \put(60,50){\line(1,0){35}}
       \put(95,50){\line(0,-1){20}}
       \put(95,30){\line(-1,0){116}}
      \end{picture}}  \lower 0pt \hbox{\,.}
\end{align*} 
\vskip-25pt \noindent
In Step 1 of $T \leftarrow 1$, we obtain $Y=T'$. Nothing happens in Step 2 since for $k=1$ we have  $t_{1,\mu'_1}\le 2\not\ge 3=x+2$. Hence $s'=s(Y)=s(T')$ and $\mu'=\mu(Y)=\mu(T')$. In the procedure $\hbox{Fram}(\mu',s')$ it is clear that the entries in the row $k>1$ will be the same as in $T$. For $k=1$, the balanced composition $(1,1,\ldots,1,2,\ldots,2)$ will not change as all entries will be at least two less the the entry directly above. Hence $T \leftarrow 1 = \hbox{Fram}(\mu',s')=T'$. For the inverse procedure,
$Y=T$ in Step 1 of $_1T'$. Again nothing happens in Step 2 since $t_{1,\mu_1}\le 2\not\ge 3$. Hence $_1T'=\hbox{Fram}(\mu,s)=T$.

\medskip
\noindent
{\bf Case 3.}  Row $k=1$ of $T$ only contains $1$'s and numbers greater than or equal to $3$. 
{}From Remark~\ref{rem:shape},  since $a_1\geq 3$, the shape of $T$ must be as follows

\vskip -10pt
\begin{align*}
T&=\begin{array}{llllll}
\vdots & & & & & \\
\scriptstyle{2k+1}& \cdots & \scriptstyle{2k+1} &  & &  \\
 \scriptstyle{2k-1}& \cdots & \scriptstyle{2k-1} & a_k &  &  \\
  \vdots & &\vdots &\vdots & & \\
\scriptstyle{2m-1}& \cdots & \scriptstyle{2m-1} & a_m & b_m & \cdots  \\
 \vdots & &\vdots &\vdots &\vdots & \\
 3 & \cdots & 3 & a_2 & b_2 & \cdots  \\
1 & \cdots & 1 & a_1 & b_1 & \cdots
 \end{array} \hskip -140pt
  \raise -90pt\hbox{ \begin{picture}(160,180)
        \put(-25,30){\line(0,1){150}}
        \put(-25,180){\line(1,0){55}}
       \put(30,180){\line(0,-1){20}}
       \put(30,160){\line(1,0){30}}
       \put(60,160){\line(0,-1){35}}
       \put(60,125){\line(1,0){25}}
       \put(85,125){\line(0,-1){35}}
       \put(85,90){\line(1,0){50}}
       \put(135,90){\line(0,-1){25}}
       \put(135,65){\line(1,0){20}}
       \put(155,65){\line(0,-1){35}}
       \put(155,30){\line(-1,0){180}}
       \put(-25,137){\line(1,0){85}}
       \put(60,125){\line(0,-1){95}}
      \end{picture}}
      = \begin{array}{cc}
         \quad\quad \raise 30pt \hbox{$A$}& \\
         \quad\quad & \\
         \quad\quad & \\
         \quad\quad & \\
         \quad\quad & \\
         \quad\quad B& \qquad \qquad \quad C\\
                  \quad\quad & \\
        \end{array}\hskip -100pt
  \raise -90pt\hbox{ \begin{picture}(160,100)
        \put(-20,30){\line(0,1){150}}
        \put(-20,180){\line(1,0){50}}
       \put(30,180){\line(0,-1){20}}
       \put(30,160){\line(1,0){30}}
       \put(60,160){\line(0,-1){35}}
       \put(60,125){\line(1,0){25}}
       \put(85,125){\line(0,-1){35}}
       \put(85,90){\line(1,0){50}}
       \put(135,90){\line(0,-1){25}}
       \put(135,65){\line(1,0){20}}
       \put(155,65){\line(0,-1){35}}
       \put(155,30){\line(-1,0){175}}
       \put(-20,137){\line(1,0){80}}
       \put(60,125){\line(0,-1){95}}
      \end{picture}}  \lower 60pt \hbox{\,.}
\end{align*}
\vskip -20pt\noindent
We use $A, B, C$ to denote the corresponding portion of $T$. Notice that $C$ has the same structure as in Case 1.   
When we insert $1$ in Step 1 of $T\leftarrow 1$, the tableau $Y$ is obtained by inserting $1$ in $C$ and the  first column of $C$ is shifted up.
In Step 2, as in Case 1, the loop runs for $k=1$ to $m$. All the entries in the portion $B$ of the tableau are set back to their current values, hence left unchanged.
Only the entries in $C$ are affected. In conclusion, this case reduces to Case 1. The same argument applies for the reverse procedure where the loop in Step 2 may run but no entries will be changed.

\medskip
\noindent
{\bf Case 4.}  Row $k=1$ of $T$ contains $2$'s, and possibly some  $1$'s, together with numbers greater than or equal to $4$. 
Again from Remark~\ref{rem:shape}, since  $a_1\geq 4$, the shape of $T$ must be as follows:
$$
T=\begin{array}{lllllllll}
\vdots & &\vdots & & \vdots& & & & \\
\scriptstyle{2r+1}& \cdots & \scriptstyle{2r+1} & \scriptstyle{2r+2} & \cdots & \scriptstyle{2r+2} & & &\\
 \scriptstyle{2r-1}& \cdots & \scriptstyle{2r-1} & \scriptstyle{2r} &\cdots & \scriptstyle{2r}&a_{r} & & \\
  \vdots & &\vdots & \vdots & &\vdots &\vdots & & \\
\scriptstyle{2m-1}& \cdots & \scriptstyle{2m-1} & \scriptstyle{2m} & \cdots & \scriptstyle{2m}  & a_m & b_m & \cdots  \\
 \vdots & &\vdots & \vdots& &\vdots &\vdots &\vdots & \\
 3 & \cdots & 3 & 4 & \cdots & 4 & a_2 & b_2 & \cdots  \\
1 & \cdots & 1 & 2 & \cdots & 2 & a_1 & b_1 & \cdots
 \end{array}\hskip -140pt
  \raise -90pt\hbox{ \begin{picture}(160,180)
        \put(-110,30){\line(0,1){150}}
        \put(-110,180){\line(1,0){100}}
       \put(-10,180){\line(0,-1){20}}
       \put(-10,160){\line(1,0){70}}
       \put(60,160){\line(0,-1){35}}
       \put(60,125){\line(1,0){25}}
       \put(85,125){\line(0,-1){35}}
       \put(85,90){\line(1,0){50}}
       \put(135,90){\line(0,-1){25}}
       \put(135,65){\line(1,0){20}}
       \put(155,65){\line(0,-1){35}}
       \put(155,30){\line(-1,0){265}}
       \put(-110,137){\line(1,0){170}}
       \put(60,125){\line(0,-1){95}}
      \end{picture}} \lower 60pt \hbox{\,.}\\ 
$$ \vskip -30pt \noindent
In Step 1 of $T\leftarrow 1$ we get
$$
Y =\begin{array}{lllllllll}
\vdots & &\vdots & &\vdots & & & & \\
\scriptstyle{2r+1}& \cdots & \scriptstyle{2r+1} & \scriptstyle{2r+2} & \cdots & \scriptstyle{2r+2} &a_r & &\\
 \scriptstyle{2r-1}& \cdots & \scriptstyle{2r-1} & \scriptstyle{2r} &\cdots & \scriptstyle{2r}&a_{r-1} & & \\
  \vdots & &\vdots & \vdots & &\vdots &\vdots & & \\
   \scriptstyle{2m+1}& \cdots & \scriptstyle{2m+1} & \scriptstyle{2m+2} &\cdots & \scriptstyle{2m+2}&a_{m} & & \\
\scriptstyle{2m-1}& \cdots & \scriptstyle{2m-1} & \scriptstyle{2m} & \cdots & \scriptstyle{2m}  & a_{m-1} & b_m & \cdots  \\
 \vdots & &\vdots & \vdots& &\vdots &\vdots &\vdots & \\
 3 & \cdots & 3 & 4 & \cdots & 4 & a_1 & b_2 & \cdots  \\
1 & \cdots & 1 & 1 & \cdots & 2 &2 & b_1 & \cdots
 \end{array} \hskip -147pt
  \raise -85pt\hbox{ \begin{picture}(170,180)
        \put(-120,20){\line(0,1){160}}
        \put(-120,180){\line(1,0){110}}
       \put(-10,180){\line(0,-1){20}}
       \put(-10,160){\line(1,0){70}}
       \put(60,160){\line(0,-1){35}}
       \put(60,125){\line(1,0){30}}
       \put(90,125){\line(0,-1){45}}
       \put(90,80){\line(1,0){50}}
       \put(140,80){\line(0,-1){15}}
       \put(140,65){\line(1,0){20}}
       \put(160,65){\line(0,-1){45}}
       \put(160,20){\line(-1,0){280}}
       \put(-120,137){\line(1,0){180}}
       \put(60,125){\line(0,-1){105}}
      \end{picture}}\lower 65pt \hbox{\,.}\\ 
$$ 
\vskip -20pt\noindent
We have that 
\begin{align*}
  s(Y)=&(s_1(Y),s_2(Y),\ldots ,s_l(Y))\\
      =&(s_1-a_{1}+1,s_2-a_{2}+a_{1},\ldots,s_r-a_{r}+a_{r-1},s_{r+1}+a_{r},s_{r+2}\ldots,s_{l})
\end{align*}
and  $\mu'=\mu(Y)=(\mu_1,\ldots,\mu_r,\mu_{r+1}+1,\mu_{r+2},\ldots,\mu_l)$. Let $[y_{i,j}]=Y$ and let $c$ be the index  of the column   with the $a_i$'s. That is $y_{1, c}=2$ and $y_{i,c}=a_{i-1}$ for $2\le i\le r+1$. Also let $c'<c$ be the index of the column where $1$ is inserted. That is $y_{1,c'}=1$ and $y_{i,c'}=2i$ for $2\le i\le r+1$.
In Step 2 of $T\leftarrow 1$, as $k$ varies from $1$ to $m$, we have $x=2k-1$ and since $b_1\ge 4$ we have $y_{k,\mu_k}\ge  b_k \geq b_1+2k-2\ge 2k+2> x+2$.
The entries in column $c$ are sequentially changed to $y_{k+1,c}:=2k+2$ since the entry in row $k$ is $y_{k,c}=2k=x+1$. Also the entries in column $c'$ are sequentially changed to $y_{k+1,c'}:=2k+1$ since the entry in row $k$ is $y_{k,c'}=2k-1=x$. The other entries in columns $1\le j\le c'-1$ and $c'+1\le j\le c-1$ are set back to their current values, hence left unchanged.
No other entries are changed since for $j\ge c+1$ we have $y_{k,j}\geq b_k\geq 2k+2>x+1$. 
The loop stops after $k=m$ since for row $m+1$, we have $x=2m+1$ and at that moment, $y_{m+1,\mu'_{m+1}}=2m+2=x+1\not\ge x+2$.
 We then have that
$d_2=a_{1}-3$, $\ldots,$ $d_{m+1}=a_{m}-(2m+1)$, $d_{m+2}=\cdots=d_l=0$. Thus the new row sum is 
\begin{equation}\label{eq:sTp} 
\begin{array}{r}
s'=(s_1-2,\ldots,s_m-2,s_{m+1}+2m+1-a_{m+1},s_{m+2}+a_{m+1}-a_{m+2},\ldots\quad \\
 \ldots, s_r+a_{r-1}-a_{r},s_{r+1}+a_r,s_{r+2},\ldots,s_{l}).
\end{array}
\end{equation}
We claim that $(\mu',s')$ satisfies the framing condition of Definition~\ref{def:framcond}. In the insertion algorithm every $y_{i,j}$ should satisfy $y_{i,j}\geq 2i-1$, and the value change in step $2$ still guarantees $s_i(Y)-d_i\geq(2i-1)\mu_i'$. These imply $s'_i\geq (2i-1)\mu_i'$ for all $1\leq i\leq l$. For $1\leq i\leq m-1$ and $\mu_{i+1}'=\mu_i'$, we have $s'_{i+1}-s'_i=(s_{i+1}-2)-(s_i-2)=s_{i+1}-s_i\geq 2\mu_i=2\mu_i'$. For $m+1\leq i\leq r$, since $s_{i+1}-a_{i+1}-(s_{i}-a_{i})\geq 2(\mu_i-1)=2(\mu_i'-1)$, we get $s'_{i+1}-s'_i=s_{i+1}+a_{i}-a_{i+1}-(s_{i}+a_{i-1}-a_{i})\geq 2(\mu_i'-1)+2=2\mu_i'$. Remember that $(\mu,s)$ satisfies the framing condition. Thus together with the case when $i\geq r+2$, $s'_i=s_i$, $\mu_i'=\mu_i$, we obtain that $(\mu',s')$ also satisfies the framing condition. We can now compute $\hbox{Fram}(\mu',s')$ in Step 3 of $T\leftarrow 1$. Notice that row $m+1$ of $Y$ after Step 2 contains only $2m+1$'s and $2m+2$'s and it is 
 already balanced. Moreover the entries in the row just above will be at least two more. This implies that row $m+1$ of $\hbox{Fram}(\mu',s')$  contains only $2m+1$'s and $2m+2$'s, which determines uniquely the numbers below and preserves all the $1$'s and $2$'s in $T\leftarrow 1$.
We get
\begin{equation}\label{eq:Tp}
T'=T\leftarrow 1 =\begin{array}{lllllllll}
\vdots & &\vdots & &\vdots & & & & \\
\scriptstyle{t'_{r+1,1}}& \cdots &  &  & \cdots &  &a'_r & &\\
 \scriptstyle{t'_{r,1}}& \cdots &   &   &\cdots &  &a'_{r-1} & & \\
  \vdots & &\vdots & \vdots & &\vdots &\vdots & & \\
   \scriptstyle{2m+1}& \cdots & \scriptstyle{2m+1} & \scriptstyle{2m+1} &\cdots & \scriptstyle{2m+2}& \scriptstyle{2m+2} & & \\
\scriptstyle{2m-1}& \cdots & \scriptstyle{2m-1} & \scriptstyle{2m-1} & \cdots & \scriptstyle{2m}  & \scriptstyle{2m} & b'_m & \cdots  \\
 \vdots & &\vdots & \vdots& &\vdots &\vdots &\vdots & \\
 3 & \cdots & 3 & 3 & \cdots & 4 & 4 & b'_2 & \cdots  \\
1 & \cdots & 1 & 1 & \cdots & 2 &2 & b'_1 & \cdots
 \end{array} \hskip -143pt
  \raise -85pt\hbox{ \begin{picture}(100,180)
        \put(-120,20){\line(0,1){160}}
        \put(-120,180){\line(1,0){110}}
       \put(-10,180){\line(0,-1){20}}
       \put(-10,160){\line(1,0){70}}
       \put(60,160){\line(0,-1){35}}
       \put(60,125){\line(1,0){30}}
       \put(90,125){\line(0,-1){45}}
       \put(90,80){\line(1,0){50}}
       \put(140,80){\line(0,-1){15}}
       \put(140,65){\line(1,0){20}}
       \put(160,65){\line(0,-1){45}}
       \put(160,20){\line(-1,0){280}}
       \put(-120,137){\line(1,0){180}}
       \put(60,125){\line(0,-1){105}}
      \end{picture}}
\end{equation} \vskip -10pt\noindent
where $b'_i\geq b_i\geq 2i+2$, for $1\leq i\leq m$. Also $t'_{m+2,j}\ge 2m+4$ for $c'\le j\le c$ and $a'_{m+1}\ge 2m+4$.

Now, we want to  show $_1T'=T$. In Step 1 of $_1T'$, the dotted box $(i,j)=(1,1)$ travels right on the first line  to $(1,c')$, then up the column to $(m+1,c')$, then right along row $m+1$ to $(m+1,c)$ and up that column to the end. We get
\begin{equation}\label{eq:Y}
Y =\begin{array}{lllllllll}
\vdots & &\vdots & &\vdots & & & & \\
\scriptstyle{t'_{r+1,1}}& \cdots &  &  & \cdots &  & & &\\
 \scriptstyle{t'_{r,1}}& \cdots &   &   &\cdots &  &a'_r & & \\
  \vdots & &\vdots & \vdots & &\vdots &\vdots & & \\
   \scriptstyle{2m+1}& \cdots & \scriptstyle{2m+1} & \scriptstyle{2m+2} &\cdots & \scriptstyle{2m+2}&a'_{m+1} & & \\
\scriptstyle{2m-1}& \cdots & \scriptstyle{2m-1} & \scriptstyle{2m} & \cdots & \scriptstyle{2m}  &  \scriptstyle{2m+1} & b'_m & \cdots  \\
 \vdots & &\vdots & \vdots& &\vdots &\vdots &\vdots & \\
 3 & \cdots & 3 & 4 & \cdots & 4 & 5 & b'_2 & \cdots  \\
1 & \cdots & 1 & 2 & \cdots & 2 &3 & b'_1 & \cdots
 \end{array} \hskip -145pt
  \raise -85pt\hbox{ \begin{picture}(100,180)
        \put(-120,20){\line(0,1){160}}
        \put(-120,180){\line(1,0){110}}
       \put(-10,180){\line(0,-1){20}}
       \put(-10,160){\line(1,0){70}}
       \put(60,160){\line(0,-1){35}}
       \put(60,125){\line(1,0){30}}
       \put(90,125){\line(0,-1){45}}
       \put(90,80){\line(1,0){50}}
       \put(140,80){\line(0,-1){15}}
       \put(140,65){\line(1,0){20}}
       \put(160,65){\line(0,-1){45}}
       \put(160,20){\line(-1,0){280}}
       \put(-120,137){\line(1,0){180}}
       \put(60,125){\line(0,-1){105}}
      \end{picture}} 
\end{equation} \vskip -10pt\noindent
We get that $\mu'=\mu(Y)=\mu$. Let $[y_{i,j}]=Y$ and we have that $y_{1,1}\le 2$. 
 In Step 2, for $k=1$ to $m$,
we have $x=2k-1$ and the conditions $y_{k,\mu'_k}\ge b'_k\ge b_k\ge 2k+2>x+2$ and $y_{k,1}\le 2k\le x+1$ hold. The loop sets $y_{k+1,j}$ in column $1\le j\le c-1$
to the same values, so no change occurs here. That is $d_2=\cdots =d_{m+1}=0$. Now for $k=m+1$ to $r$ we have that conditions $y_{k,\mu'_k}= a'_k\ge a'_{m+1}+2(k-m-1)\ge 2k+2>x+2$ and $y_{k,1}=2k-1<x+1$ hold. The loop sets $y_{k+1,j}=2k+1$ for $1\le j\le c'-1$ and $y_{k+1,j}=2k+2$ for $c'\le j\le c-1$. The loop stops after $k=r$ since $y_{r+1,\mu_{r+1}}=2r+2=x+1\not\ge x+2$. We have $d_{r+2}=\cdots=d_l=0$ and for $m+1\le k\le r$ we have
\begin{align*}
d_{k+1}&=t'_{k+1,1}+\cdots+t'_{k+1,c-1}-(c'-1)(2k+1)-(c-c')(2k+2)\\
              &=a_k-a'_k\,.
\end{align*}
The second equality follows from comparing the $k+1^{th}$ entry of $s(T')$ in Eq~(\ref{eq:sTp}) with the row sum of $s(T')$ in the framed tableau in Eq~(\ref{eq:Tp}). We also remind the reader that from the start, row $k+1$ of $T$ is such that $s_{k+1}-a_{k+1}=(c'-1)(2k+1)+(c-c')(2k+2)$.
The row sum $s(Y)$ for $Y$ in Eq~(\ref{eq:Y}) is obtained from  $s(T')$ in Eq~(\ref{eq:sTp}). We have
 $$ s(Y)=(s_1,\ldots,s_m,s_{m+1}+a'_{m+1}-a_{m+1},s_{m+2}+a_{m+1}-a_{m+2}-a'_{m+1}+a'_{m+2},\ldots$$
 $$\qquad\qquad \qquad\qquad\qquad  \ldots, s_r+a_{r-1}-a_{r}-a'_{r-1}+a'_{r},s_{r+1}+a_r-a'_r,s_{r+2},\ldots,s_{l}).$$
The expression for $s'$ at the end of Step 2 in the procedure $_1T'$ then gives us $s'=s=s(T)$ and then $_1T'=\hbox{Fram}(\mu,s)=T.$

\medskip
\noindent
{\bf Case 5.}   Row $k=1$ of $T$ contains $2$'s and $3$'s, possibly some  $1$'s and possibly some numbers greater than or equal to $4$. 
Depending on the numbers appearing in the first row of $T$, we have
$$
T=\ \begin{array}{l}
\, \vdots \hskip 24pt \vdots \\
3\,  \cdots\,  3 \  4\,  \cdots\,  4   \hskip 12pt    \cdots \hskip 20pt \cdots \\
1\,  \cdots\,  1 \  2\,  \cdots \, 2 \  3\,  \cdots \, 3 \  c\, \cdots\\ 
\end{array}\hskip -150pt
  \raise -20pt\hbox{ \begin{picture}(150,70)
        \put(0,0){\line(0,1){70}}
        \put(0,70){\line(1,0){30}}
        \put(30,70){\line(0,-1){20}}
        \put(30,50){\line(1,0){46}}
        \put(76,50){\line(0,-1){20}}
        \put(76,30){\line(1,0){70}}
        \put(146,30){\line(0,-1){30}}
        \put(146,0){\line(-1,0){146}}
      \end{picture}}
 \hbox{\quad or \quad}
     \begin{array}{l}
\, \vdots \hskip 24pt \vdots \\
3\,  \cdots\,  3 \  c\,  \cdots\,  c'   \hskip 10pt    \cdots  \\
1\,  \cdots\,  1 \  2\,  \cdots \, 2 \  3\,  \cdots \, 3 \\ 
\end{array}\hskip -122pt
  \raise -20pt\hbox{ \begin{picture}(125,70)
        \put(0,0){\line(0,1){70}}
        \put(0,70){\line(1,0){30}}
        \put(30,70){\line(0,-1){20}}
        \put(30,50){\line(1,0){46}}
        \put(76,50){\line(0,-1){20}}
        \put(76,30){\line(1,0){45}}
        \put(121,30){\line(0,-1){30}}
        \put(121,0){\line(-1,0){121}}
      \end{picture}}
$$
where $c,c'\geq 4$. If there is no $c$ in the first row, then from Lemma~\ref{lem:framcar}, we are not forced to have $4$ above the $2$'s in the first row and there is thus no restrictions on the numbers above those 2's.
We  use induction on the length $l$ of $T$ to prove this case.
For $l(T)=1$, it is easy to check that all the procedures are well defined and $_1(T\leftarrow 1)=T$.  Assume that the result is true up to $l(T)=n$, and $T\leftarrow 1$  preserves the added $1$ and all the $1$'s and $2$'s in $T$. That is the $1$'s and $2$'s of  $Y$ in Step 1 of  $T\leftarrow 1$ are left unchanged in the remaining steps. This was the situation in Cases 1--4 above.
For $l(T)=n+1$, let $R$ denote the first row of $T$ and $T_2$ denote the remaining tableau. That is $T_2$ consists of rows 2 and up of $T$.
{}From Lemma~\ref{lem:TopT} we know that $T_2$ is a framed tableau of length $n$.
In Step 1 of $T\leftarrow 1$, to get $Y$,
 we insert $1$ in $R$. We then have that $3$ is bumped up and inserted in $T_2$. Denote by $Y_2$ the result of Step 1 of $T_2\leftarrow 3$.
Clearly $Y_2$ is also the tableau we get from rows 2 and up of $Y$. We have
  $$T= \ \begin{array}{l}  T_2\\ R  \end{array}
  \hskip -25pt  \raise -10pt\hbox{ \begin{picture}(40,40)
        \put(0,0){\line(0,1){40}}
        \put(0,40){\line(1,0){10}}
        \put(10,40){\line(0,-1){10}}
        \put(10,30){\line(1,0){15}}
        \put(25,30){\line(0,-1){10}}
        \put(25,20){\line(1,0){10}}
        \put(35,20){\line(0,-1){20}}
        \put(35,0){\line(-1,0){35}}
        \put(0,13){\line(1,0){5}} \put(10,13){\line(1,0){5}} \put(20,13){\line(1,0){5}} \put(30,13){\line(1,0){5}}
      \end{picture}}
  \qquad\hbox{ and }\qquad Y= \ \begin{array}{l}  Y_2\\ 1\,R'  \end{array}
    \hskip -32pt  \raise -10pt\hbox{ \begin{picture}(40,40)
        \put(0,0){\line(0,1){40}}
        \put(0,40){\line(1,0){10}}
        \put(10,40){\line(0,-1){10}}
        \put(10,30){\line(1,0){15}}
        \put(25,30){\line(0,-1){10}}
        \put(25,20){\line(1,0){10}}
        \put(35,20){\line(0,-1){20}}
        \put(35,0){\line(-1,0){35}}
        \put(0,13){\line(1,0){5}} \put(10,13){\line(1,0){5}} \put(20,13){\line(1,0){5}} \put(30,13){\line(1,0){5}}
      \end{picture}}  \lower 10pt \hbox{\,.}\\ 
$$
\medskip

\noindent
 It is important to remark that the number of $1$'s in the first row of $Y$ is exactly the number of $3$'s 
in the first row of $Y_2$. 
In Step 2 of $T\leftarrow 1$, for $k=1$ we have $y_{1,\mu_1}\ge 3\ge x+2$. In the case when there are numbers $c\ge 4$ in the first row of $T$, we must have $4$ above each $2$ in the first row. The first loop of Step 2 just sets all values $3$ and $4$ back to the same values. Hence $d_2=0$ in this case. If there are only $1$'s, $2$'s and $3$'s in the first row of $T$, then there is no restriction above the $2$'s. But in this case, we have $y_{1,\mu_1}= 3\not> x+2$ and no number above the $2$'s changes. Hence in all cases $d_2=0$.
The remaining loops of Step 2 of $T\leftarrow 1$ are identical to Step 2 for $T_2\leftarrow 3$. By the induction hypothesis and Lemma~\ref{lem:red}, $T_2\leftarrow 3$
is well defined and gives a framed tableau $T_2'=T_2\leftarrow 3$ such that all the $3$'s and $4$'s in the first line are the same as $Y_2$. The shape $\mu'=\mu(Y)=(\mu_1,\mu'_2,\ldots,\mu'_{l'})$ where
$(\mu'_2,\ldots,\mu'_{l'})=\mu(Y_2)=\mu(T'_2)$. Also $s'=(s_1-2,s'_2,\ldots,s'_{l'})$ where 
$(s'_2,\ldots,s'_	{l'})=s(T'_2)$. It is clear, by definition, that $(\mu(T'_2),s(T'_2))$ satisfies the framing condition. In fact, since the smallest entry of $T_2$ is $3$, we also have that $T_2-2$ is a framed tableau. This implies that $s'_i\ge (2i-1)\mu'_j$ for $2\le i\le l'$.
Clearly, $s_1-2\ge \mu_1$, so we only need to verify Condition~2 of Definition~\ref{def:framcond} for $i=1$.
If $\mu_1>\mu'_2$, then there is nothing to check. By Cases 1--4 and by induction, we remark that $s'_1\ge s_1-2$. This implies that for $T'_2$, we have $s'_2\ge s_2-2$.
Hence if $\mu_1=\mu_2=\mu'_2$, then $s'_2\ge s_2-2\ge s_1+2\mu_1-2 = (s_1-2)+2\mu_1$. We are left to consider the case where $\mu_1=\mu'_2=\mu_2+1$.
This may only happen if all the entries in the second row of $T$ are only $3$'s and $4$'s and in this case
$$
T\leftarrow 1 =\ \begin{array}{l}
\, \vdots \hskip 24pt \vdots \\
3\,  \cdots\,  3 \  4\,  \cdots\,  4   \\
1\,  \cdots\,  1 \  2\,  \cdots \, 2 \  3\\ 
\end{array}\hskip -93pt
  \raise -20pt\hbox{ \begin{picture}(90,70)
        \put(0,0){\line(0,1){70}}
        \put(0,70){\line(1,0){30}}
        \put(30,70){\line(0,-1){20}}
        \put(30,50){\line(1,0){46}}
        \put(76,50){\line(0,-1){36}}
        \put(76,14){\line(1,0){10}}
        \put(86,14){\line(0,-1){14}}
        \put(86,0){\line(-1,0){86}}
      \end{picture}}
\leftarrow 1 = \ 
     \begin{array}{l}
\, \vdots \hskip 24pt \vdots \\
3\,  \cdots\,  3 \  3\,  \cdots\,  4 \  4 \\
1\,  \cdots\,  1 \  1\,  \cdots \, 2 \  2 \\ 
\end{array}\hskip -93pt
  \raise -20pt\hbox{ \begin{picture}(90,70)
        \put(0,0){\line(0,1){70}}
        \put(0,70){\line(1,0){30}}
        \put(30,70){\line(0,-1){20}}
        \put(30,50){\line(1,0){46}}
        \put(76,50){\line(0,-1){36}}
        \put(76,14){\line(1,0){10}}
        \put(86,14){\line(0,-1){14}}
        \put(86,0){\line(-1,0){86}}
      \end{picture}} \lower 20pt \hbox{\,.}\\ 
$$
By induction, the entries in the second row are $3$'s and $4$'s.
Clearly $s'_2\ge (s_1-2)+2\mu_1$. We have that in all cases $(\mu',s')$ satisfy the framing condition and we get a well defined framed tableau $T'=\hbox{Fram}(\mu',s')=T\leftarrow 1$. All the $1$'s and $2$'s in the first row of $Y$ are preserved in $T'$.

Now we consider the procedure $_1T'$. Let $T'_2$ be the framed tableau formed by rows 2 and up of $T'$. In Step 1, we get a tableau $Y$  with a 1 replaced  by a 3 in the first row of $T'$, and a tableau $Y_2$ in rows 2 and up of $Y$. Again it is clear that $Y_2$ is the same as the one obtained in Step 1 of $_3T'_2$. In Step 2, for $k=1$, we have $y_{1,1}\le 2=x+1$ and $y_{1,\mu_1}\ge 3=x+2$. The same argument as above shows that $d_2=0$. The remaining loops of Step 2 of $_1T'$ are the same as Step 2 in $_3T'_2$. By the induction hypothesis and Lemma~\ref{lem:red}, we know that $_3T'_2=T_2$ is well defined and gives rows 2 and up of $T$.
The first row sum of $Y$ is now $s_1$, so at the end of Step 2 we have $s'=s(T)$. Also for $\mu'=\mu(Y)$, we clearly have $\mu'_1=\mu_1$ and by the induction hypothesis $\mu'_i=\mu_i$ for $i\ge 2$. Hence we get $T=\hbox{Fram}(\mu,s)=_1T'$. This proves Case 5.

Let $\mathcal{F}_{k,l}=\{T\hbox{ framed tableau}: \mu(T)\vdash l, s(T)\models k\}$ and let $\mathcal{P}_{k,l}=\{\lambda=(\lambda_1,\ldots,\lambda_l)\vdash k\}$.
So far, we have that $T= {_x(T\leftarrow x)}$ for all $T\in\mathcal{F}_{k,l}$ and $x$. 
This implies that the map $(T,x) \mapsto (T\leftarrow x)$ is injective. 
We have an injection $\mathcal{P}_{k,l}\hookrightarrow \mathcal{F}_{k,l}$ defined by $\lambda\mapsto (\emptyset\leftarrow\lambda)$.
Let us pick $n>k$ and consider $\{F_T\Delta_n: T\in\mathcal{F}_{k,l}\}\subset A_n^{k,l}$. For $T\in\mathcal{F}_{k,l}$, let  $(\mu_1,\ldots,\mu_r)=\mu(T)$, $(s_1,\ldots,s_r)= s(T)$ and $F_{T}$ defined as in Eq~(\ref{eq:Fop}).
Iterating Remark~\ref{rem:lead}, we get 
\begin{align}\label{eq:LTCOM}
F_T\Delta_n =&F_{T_{\mu_1}}F_{T_{\mu_1-1}}\cdots F_{T_2}F_{T_1}\Delta_n  \nonumber \\
            =&F_{T_{\mu_1}}F_{T_{\mu_1-1}}\cdots F_{T_2}(\Delta_{E^{f^o_1}_{T_1}[(0,0),(1,0),\ldots ,(n-1,0)]}+lower\ terms)  \nonumber \\
            =&F_{T_{\mu_1}}F_{T_{\mu_1-1}}\cdots F_{T_3}(\Delta_{E^{f^o_2}_{T_2}E^{f^o_1}_{T_1}[(0,0),(1,0),\ldots ,(n-1,0)]}+lower\ terms)  \nonumber \\
            =&\cdots \nonumber \\
            =&\Delta_{E^{f^o_{\mu_1}}_{T_{\mu_1}}\circ \cdots\circ E^{f^o_1}_{T_1}[(0,0),(1,0),\ldots ,(n-1,0)]}+lower\ terms ,
\end{align}
where $f^o_j\colon\{1,\ldots,\mu^t_j\}\rightarrow\{1,\ldots,n\}$, $1\leq j\leq \mu_1$ are defined by $f^o_j(i)= i$ for all $1\leq i\leq \mu^t_j$. So  we have $$E^{f^o_{\mu_1}}_{T_{\mu_1}}\circ \cdots\circ E^{f^o_1}_{T_1}[(0,0),(1,0),\ldots ,(n-1,0)]=[(0,0),\ldots,(n-r-s_r,\mu_r),\ldots,(n-1-s_1,\mu_1)].$$
Since $n>k$, thus we have $\Delta_{[(0,0),\ldots,(n-r-s_r,\mu_r),\ldots,(n-1-s_1,\mu_1)]}\neq 0$,  which gives the leading diagram of $F_T\Delta_n$.
Proposition~\ref{prop:inj} gives us that for different framed tableaux $T$ we get different pairs $(\mu,s)$, hence different leading terms for $F_T\Delta_n$. This gives us that the set $\{F_T\Delta_n: T\in\mathcal{F}_{k,l}\}\subset A_n^{k,l}$ is linearly independent. Recall that the dimension of $A_n^{k,l}$ is the coefficient of $q^k t^l$ in $\widetilde{C}_n(q,t)$. We claim that this coefficient is equal to $\left| \mathcal{P}_{k,l}\right|$. Indeed for $k<n$, we have that any partition $\lambda\in\mathcal{P}_{k,l}$ satisfy $\lambda_1= k-\lambda_2-\cdots-\lambda_l\le k-l+1<n-l+1$. For $k<n$, if we consider $\mu=\lambda^t\in\mathcal{P}_{k,l}$ as in Remark~\ref{rem:lambda}, then we have a bijection between $\lambda\in\mathcal{P}_{k,l}$ and the Catalan paths with coarea equal to $k$ and a single bounce $l$.
This gives
  $$\left| \mathcal{P}_{k,l}\right| \le \left| \mathcal{F}_{k,l} \right| \le \dim A_n^{k,l} = \left| \mathcal{P}_{k,l}\right|,$$
  and we must have equality everywhere. This shows that the map $(T,x) \mapsto (T\leftarrow x)$ must be surjective. Hence $_xT$ is well defined everywhere and inverse to $T\leftarrow x$.
\end{proof}

The computation in Eq~(\ref{eq:LTCOM}) shows the following:
\begin{corollary}\label{cor:basis}
Let $\widetilde{C}_{n,k,l}$ be the coefficients of $q^kt^l$ in $\widetilde{C}_n(q,t)$. We have:

1. If $k<n$, then $\widetilde{C}_{n,k,l}$ is the number of partitions of $k$ into $l$ parts;

2. There exists a natural map $\lambda\mapsto F_{\emptyset\leftarrow \lambda}$ between partitions and $F$-operators such that if $|\lambda|<n$, then the set of polynomials $\{F_{\emptyset\leftarrow \lambda}\Delta_n: \lambda \in\mathcal{P}_{k,l}\}$ forms a basis of the space $A_n^{k,l}$. 
\end{corollary}

\begin{remark} When $k\geq n$ the leading diagram for $F_{\emptyset\leftarrow \lambda}$ is not necessarily given by Remark~\ref{rem:lead}. This complicates the investigation of finding a basis for those cases. We were successful in finding bases for any $k$ and $l=3$, but the analyses is much more complicated.
\end{remark}

\begin{remark} We presented the work here with the perspective of finding a basis for $A_n^{k,l}$. But the combinatorics of the bijection between partitions and framed tableaux via $T\leftarrow x$ and $_xT$ could be very interesting in their own right and have different applications. 
\end{remark}


\parindent=0pt

\begin{thebibliography}{XX}

\bibitem{GHag}
{\sc A.~M. Garsia and J. Haglund}, {\em A proof of the $q,t$-Catalan positivity conjecture}, Discrete Math.  256  (2002),  no. 3, 677--717.

\bibitem{Haim}
{\sc M. Haiman}, {\em Vanishing theorems and character formulas for the Hilbert scheme of points in the plane}, 
  Invent. Math.  149  (2002),  no. 2, 371--407.

\bibitem{Jim}
{\sc J. Haglund}, {\em The q,t-Catalan Numbers and the Space of Diagonal Harmonics}, AMS University Lecture Series, Vol. 41(2008) 167pp.

\bibitem{Mac} I.G. Macdonald,  {\em Symmetric Functions and Hall-Polynomials, Oxford Mathematical Monographs}, Oxford Univ. Press, second edition (1995) 488p.


\end{thebibliography}
\end{document}